\documentclass[journal,twoside,web]{ieeecolor}
\usepackage{generic}

\usepackage{amsmath,amssymb,amsfonts}

\usepackage{graphicx}
\usepackage{mathtools}

\usepackage{multirow}
\usepackage{booktabs}
\usepackage{siunitx}

\usepackage{makecell}

\usepackage[ruled, vlined, linesnumbered]{algorithm2e}
\SetKwRepeat{Do}{do}{while}%
\SetKwComment{Comment}{$\triangleright$\ }{}
\DontPrintSemicolon
\SetNlSty{}{}{:} 
\SetAlgoNlRelativeSize{-1} 
\SetCommentSty{} 

\SetKwInOut{KwInit}{Init}

\usepackage{etoolbox} 
\makeatletter
\patchcmd{\@algocf@start}{-1.5em}{0em}{}{} 
\makeatother

\usepackage{flushend}
\usepackage{textcomp}
\def\BibTeX{{\rm B\kern-.05em{\sc i\kern-.025em b}\kern-.08em
    T\kern-.1667em\lower.7ex\hbox{E}\kern-.125emX}}
\usepackage{fancyhdr}
\fancypagestyle{firstpage}{
    \fancyhf{}
    \fancyfoot[C]{\small This work has been submitted to the IEEE for possible publication. \\Copyright may be transferred without notice, after which this version may no longer be accessible.}
}

\usepackage[bibstyle=ieee,citestyle=numeric-comp,doi=false,isbn=false,eprint=false]{biblatex}
\newtheorem{theorem}{Theorem}
\newtheorem{proposition}{Proposition}
\newtheorem{definition}{Definition}
\newtheorem{remark}{Remark}
\newtheorem{assumption}{Assumption}
\newtheorem{lemma}{Lemma}
\newtheorem{corollary}{Corollary }

\usepackage{hyperref}
\hypersetup{
    colorlinks=true,
    linkcolor=blue,
    citecolor=blue,
    urlcolor=blue,
    bookmarksnumbered=true,
    bookmarksopen=true
}

\begin{document}

\title{Data-Driven Continuous-Time LQR: \\ Linking Reinforcement Learning and \\ Behavioral System Theory}

\author{Armin Gießler, \IEEEmembership{Member, IEEE}, Felix Thömmes, \IEEEmembership{Member, IEEE}, \\ and Sören Hohmann, \IEEEmembership{Senior Member, IEEE}
\thanks{A. Gießler, F. Thömmes, and S. Hohmann are with the Institute of Control Systems, Karlsruhe Institute of Technology, 76131 Karlsruhe, Germany (e-mail: \{armin.giessler, felix.thoemmes, soeren.hohmann\}@kit.edu).}}

\maketitle
\thispagestyle{firstpage}

\begin{abstract}
Behavioral system theory and reinforcement learning (RL) have independently developed data-driven approaches to the continuous-time LQR problem, yet their respective methods have largely evolved in separation. We connect these two adjacent lines of work by studying their data-driven parameterizations, namely a closed-loop (CL) parameterization from behavioral system theory and an integral RL (IRL) parameterization. The CL parameterization characterizes the closed-loop system via a matrix that satisfies equality constraints. While this parameterization has been extensively studied for discrete-time systems, we adapt key results to the continuous-time setting and develop a policy iteration (PI) scheme, derive an algebraic Riccati equation, and introduce an alternative convex program. The IRL parameterization utilizes off-policy data to perform policy evaluation, which is used for PI or value iteration. Within the IRL framework, we derive a policy gradient flow and propose convex programs for the LQR problem. Finally, we provide a unified treatment of these parameterizations that enables a systematic understanding of existing approaches, and clarifies their structural relationships.
\end{abstract}

\begin{IEEEkeywords}
Linear quadratic regulator (LQR), policy optimization (PO), data-driven control, behavioral system theory, policy gradient, integral reinforcement learning
\end{IEEEkeywords}

\section{Introduction}
\IEEEPARstart{I}{n} the past two decades, a variety of promising data-driven control approaches have emerged, 
deriving the control law directly from measured data without explicitly identifying the underlying system dynamics. 
Given meaningful input-state data, these methods synthesize controllers to achieve stabilization or performance optimization.
Broadly, they can be categorized into Reinforcement Learning (RL) approaches \cite{sutton2018reinforcement,recht2019tour,vrabie2012optimal,bucsoniu2018reinforcement} and behavioral system theory approaches \cite{markovsky2023data,MARKOVSKY202142,Waarde2025}. 
In RL, the control policy is learned to maximize cumulative rewards via model-free schemes, such as value, Q\nobreakdash-, or policy iteration. The behavioral framework characterizes system behavior entirely from data, enabling model-based control methods to be reformulated in the notation of data.
Despite a cultural gap that has historically kept these communities separate, both frameworks share strong conceptual links in the context of data-driven LQR, as discussed herein. Furthermore, each framework lacks certain complementary approaches that are developed in this work.

\subsubsection*{Literature Review}
This review focuses on infinite-horizon data-driven linear quadratic regulator (LQR) methods, beginning with behavioral approaches. In \cite{de2019formulas}, the closed-loop dynamics of a discrete-time linear system with state feedback are parameterized by a matrix satisfying a matrix equality. 
When state-derivative measurements are available, this closed-loop (CL) parameterization extends to continuous-time systems.
This representation laid the foundation for many subsequent data-driven methods.
Convex optimization and policy gradient solutions to the discrete-time LQR problem are presented in \cite{de2019formulas} and \cite{Zhao2023}, respectively.
For continuous-time systems, \cite{LOPEZ2026113002} computes the Riccati matrix via convex optimization, followed by linear equations for the optimal feedback gain.
A policy iteration (PI) based on the CL parameterization is presented in \cite{Lopez2023}.
Robustness analysis and modification of these methods to handle noise-corrupted measurements are addressed in \cite{de2019formulas,dorfler2023,dorfler2022,DEPERSIS2021109548}. 
In \cite{ZhaoCovariance,Zhao2025}, a variant of the CL parameterization based on the sample covariance is proposed, making the parameterization independent of the number of samples.

In the context of RL for continuous-time systems, integral RL (IRL) replaces the model-based policy evaluation in PI with an integral variant, enabling learning without knowledge of the system matrix or state derivatives \cites{VRABIE2009477,Tamimi2007}[Ch.~3]{vrabie2012optimal}. 
However, the input matrix is still required in the policy improvement step. To overcome this, \cites{JIANG20122699}[Ch.~2]{jiang2017radp} propose an offline PI method for completely unknown  dynamics, which remains robust under small additive bounded disturbances \cite{pang_robust_2022}.
Online Q-learning approaches exist that require an initial stabilizing policy \cite{LEE20122850}, whereas methods in \cite{vamvoudakis2017q,Possieri2022,Possieri2023} eliminate this requirement.
For PI, this requirement is removed by homotopy-based initializations \cite{chen_homotopic_2022,Fan2025a} or an adaptive parameter \cite{ma_adaptive_2025}, and relaxed by a bias parameter \cite{jiang_bias-policy_2022}.
An offline data-driven value iteration (VI)  in \cite{BIAN2016348} similarly eliminates the need for an initial stabilizing policy, but requires specification of hyperparameters. 
The hybrid iteration in \cite{qasem_adaptive_2021} combines PI and VI, dispensing with an initial stabilizing policy while converging faster than VI.
All these RL methods utilize integrated state and input trajectories over time  intervals to avoid  state derivatives. When implemented digitally, these integrals can be approximated using sampled measurements and numerical integration techniques~\cite{Song2025}.

Another research direction employs a different parameterization through the data informativity framework, 
 which addresses cases where the data is not persistently exciting. Initially developed for discrete-time systems \cite{Waarde2025,WaardeInform,WaardeSLemma} and later extended to continuous-time systems \cite{jaap2025}, this approach indirectly parameterizes the unknown system via two quadratic matrix inequalities (QMIs). The resulting linear matrix inequality (LMI), derived via the matrix S-lemma, characterizes all systems consistent with the data and enables non-conservative controller synthesis, applicable to $H_2$ control \cite{WaardeSLemma} with the LQR problem as a special case \cite{Feron}. 
These results assume noise satisfying a quadratic inequality, a generalization of an $l_2$-energy bound, while \cite{Finsler2021} extends the framework to the noiseless case. 
 Using the same noise model, \cite{Berberich2020} achieves robust stabilization of discrete-time linear systems via the CL parameterization from \cite{de2019formulas}.
Quadratic stabilization of discrete-time linear systems under $l_\infty$-bounded noise is studied in \cite{DAI2023111041} using a polynomial matrix-based parameterization.
In \cite{Rapisarda2024}, orthogonal polynomial bases approximate the continuous-time dynamics through data-dependent coefficient sequences, enabling controller design without requiring state-derivative measurements.
In \cite{Silva2019}, the discrete-time LQR problem is solved via a closed-form solution of the algebraic Riccati equation,
employing a data-driven parameterization of the extended observability and Toeplitz matrices.

\subsection{Contributions}

The contributions of this work are threefold: 
\begin{enumerate}
    \item We link the data-driven LQR approaches from behavioral system theory and RL by providing a unified treatment of the CL and IRL parameterizations and analyzing their structural connections and differences. This establishes a systematic perspective on existing approaches and clarifies their interrelationships.
    \item Under the CL parameterization, we adapt discrete-time results to continuous time without requiring state-derivative measurements, and we 
    derive a PI scheme, a data-driven continuous-time algebraic Riccati equation (CARE), and an alternative convex problem formulation.
\item Under the IRL parameterization, we build on policy evaluation results to derive convex optimization formulations, a policy gradient flow, and a data-driven CARE.
\end{enumerate}

Table~\ref{tab:contr} summarizes the contributions of this article relative to existing approaches.
For cells with multiple components, the components correspond to algebraically different formulations within the same parameterization. 
The paper focuses on structural aspects and excludes stochastic noise.

\begin{table}[tbh]
   \centering
   \label{tab:contr}
   \caption{Comparison of the proposed contributions with existing methods under the  CL and IRL parameterizations.}
   \setlength{\tabcolsep}{2pt}
\begin{tabular}{ccc}
\toprule
\multirow{2}{*}{ \centering \textbf{Method}} & \multicolumn{2}{c}{\textbf{Parameterization}} \\
\cmidrule(lr){2-3}
& \textbf{CL} & \textbf{IRL} \\
\midrule
Policy Evaluation     & adapted  from \cite{Zhao2023}       & established in \cite{JIANG20122699} \\
Policy Improvement    & novel                               & established in \cite{JIANG20122699} \\
Policy Iteration      & novel, established in \cite{Lopez2023}                             & established in \cite{JIANG20122699} \\
CARE & novel   & novel, established in \cite{BIAN2016348}\\ 
Riccati Flow & novel  & adapted  from \cite{BIAN2016348} \\
Value Iteration       & novel                                & established in \cite{BIAN2016348} \\
\multirow{2}{*}{ Convex Opt. Problem}   & novel, adapted from \cite{de2019formulas},  & \multirow{2}{*}{ novel}  \\ 
& established in \cite{LOPEZ2026113002} & \\
Policy Gradient Flow  & adapted  from \cite{Zhao2023}        & ~novel \\
\bottomrule
\end{tabular}
\end{table}

\subsection{Paper Organization}

Preliminaries are given in Section~\ref{sec:prelim}. Section~\ref{sec:param} introduces the data-driven parameterizations and compares them. PI algorithms are presented in Section~\ref{sec:PI}, while the CARE, Riccati flow, and VI are part of Section~\ref{sec:CARE}. Section~\ref{sec:conv_prob} focuses on convex programs, and Section~\ref{sec:grad_flow} formulates policy gradient flows. Numerical results and discussion are provided in Section~\ref{sec:comp}, followed by a conclusion in the final section.

\subsection{Notation}
The set of (nonnegative) real numbers is denoted by $\mathbb{R}$ ($\mathbb{R}_{\geq 0}$), and the set of nonnegative integers is denoted by $\mathbb{N}_0$. The set of $n$-dimensional symmetric matrices is denoted by $\mathbb{S}^n$, with $\mathbb{S}_{++}^n$ ($\mathbb{S}_+^n$) denoting the subsets of symmetric positive definite (semidefinite) matrices. A symmetric positive definite (semidefinite) matrix $A$ is denoted by $A \succ 0$ ($A \succeq 0$).
The pseudoinverse of matrix $A$ is denoted by $A^\dagger$. If $A$ has full column (row) rank, the left-inverse (right-inverse) is $A^\dagger = (A^\top A)^{-1}A^\top$ $(A^\dagger = A^\top ( A A^\top)^{-1})$. Otherwise, 
$A^\dagger$ denotes the minimum-norm least-squares solution. 
The eigenvalues and trace of a square matrix $A$ are denoted by $\lambda_i(A)$ and $\operatorname{tr}(A)$, respectively. 
For a matrix $A$, $\operatorname{vec}(A)$, $\operatorname{vech}(A)$, $\operatorname{rank}(A)$, $\operatorname{ker}(A)$, and $\operatorname{im}(A)$ denote its vectorization, half-vectorization, rank, kernel, and column space, respectively. The commutation matrix $C_{n,m}$ satisfies $\operatorname{vec}(A^\top)=C_{n,m}\operatorname{vec}(A)$ for $A\in\mathbb{R}^{n\times m}$.

\section{Preliminaries}
\label{sec:prelim}
\subsection{Linear Quadratic Regulator}
\label{subsec:LQR} 
We briefly recall model-based LQR problem formulations to aid understanding of the data-driven approaches presented later.
Consider the system, referred to as $(A,B)$,
\begin{align}
 \dot x(t) = A x(t) + Bu(t), \quad x(0)=x_0, \label{math:lin_sys}
\end{align} 
where $x(t)\in\mathbb{R}^{n}$ and $u(t)\in\mathbb{R}^m$.
The LQR objective is 
\begin{align}
    f(x,u) & = \int_{0}^{\infty} x(t)^\top Q x(t) + u(t)^\top R u(t)  \mathrm{d}t, \label{math:LQR_cost}
\end{align}
where $Q\in\mathbb{S}^n_{+}$ and $R\in\mathbb{S}^m_{++}$.
The LQR problem is \cite[Sec.~3.2]{anderson2007}
\begin{align}
    \min_{u,x} ~& f(x,u)\quad \text{s.t.} \quad \eqref{math:lin_sys}. \label{math:LQR_standard}
\end{align}

\begin{assumption}
\label{ass:1}
The system $(A,B)$ is stabilizable and the pair $(A,\sqrt{Q})$ is detectable. 
\end{assumption}

Assumption~\ref{ass:1} holds throughout. 
Under Assumption~\ref{ass:1}, the optimal control $u^*=-K^*x$, with $K^*= R^{-1} B^\top P^*$, stabilizes the system \eqref{math:lin_sys}, where $P^*$ is the unique, positive semidefinite solution of the CARE \cite[Th. 5]{kuvcera1973review}
\begin{align}
    A^\top P + P A - PBR^{-1}B^\top P + Q = 0. \label{math:CARE}
\end{align}

The LQR cost \eqref{math:LQR_cost}
can be expressed as a function of the feedback gain $K$ under closed-loop dynamics $A_K\coloneqq A-BK$
\begin{align}
    f_K(x_0) & = \int_{0}^{\infty} x^\top \left(Q + K^\top R K  \right) x \mathrm{d}t \notag\\
    & = \int_0^\infty x_0^\top e^{A_K^\top t} \left( Q + K^\top R K  \right) e^{A_K t} x_0 \mathrm{d}t \notag \\
    & = x_0^\top P x_0 = \operatorname{tr }(P x_0 x_0^\top ) \label{math:boyd_obj}\\
    & = \operatorname{tr }\left( \left(Q + K^\top R K  \right) Y \right), \label{math:feron_obj}
\end{align} 
where 
\begin{align}
    P & = \int_0^\infty e^{A_K^\top t }\left(Q + K^\top R K  \right) e^{A_K t}\mathrm{d}t, \label{math:int_P} \\
    Y & =  \int_0^\infty e^{A_K t } x_0 x_0^\top e^{A_K^\top t} \mathrm{d}t. \label{math:int_Y}
\end{align}
If $A_K$ is Hurwitz,  then \eqref{math:int_P} and \eqref{math:int_Y} are finite, and $P\succeq 0$ and $Y\succeq 0$ uniquely satisfy the Lyapunov equations \cite[Th.~5.3.1]{lancaster1995algebraic}
\begin{align}
    0 & = A_K^\top P + P A_K + Q + K^\top R K, \label{math:Lyap1} \\
    0 & = A_K Y + Y A_K^\top + x_0 x_0^\top. \label{math:Lyap2}
\end{align}
Substituting $K=R^{-1} B^\top P$ into \eqref{math:Lyap1} recovers the CARE \eqref{math:CARE}. 

The LQR cost is lower bounded by the quadratic function $x_0^\top P x_0 \leq f_K(x_0)$.
The greatest lower bound $f^*=x_0^\top P^* x_0$, denoting the minimum value of \eqref{math:LQR_standard}, is attained by selecting a stabilizing feedback $K$ that satisfies \eqref{math:Lyap1} and minimizes \eqref{math:boyd_obj}. Alternatively, $f^*$ can be computed via the convex program~\cite{Balakrishnan1995}%
\begin{subequations}
\label{math:LQR_boyd}
\begin{align}
    \max_P ~& \operatorname{tr }(P) \\
    \text{s.t.}~& \begin{bmatrix} A^\top P + P A + Q & PB \\ B^\top P & R \end{bmatrix} \succeq 0, \label{math:LQR_boyd_care}\\
    & P\succeq 0.
\end{align}
\end{subequations}
The Schur complement applied to the LMI \eqref{math:LQR_boyd_care} yields the relaxed Riccati inequality $A^\top P + PA  -PBR^{-1}B^\top P + Q \succeq 0$ which encompasses the CARE solutions where the inequality saturates to equality. Maximizing $\operatorname{tr}(P)$ over this convex set attains the stabilizing solution $P^*$ of the CARE \cite[p.~115]{LMIbook}.

On the other hand, the optimal LQR objective $f^*$ is upper bounded by \eqref{math:feron_obj} for every stabilizing feedback $K$, where $Y$ satisfies \eqref{math:Lyap2}. The tightest upper bound under state-feedback is characterized by $K^*$ and $Y^*$ solving the problem\footnote{The problem \eqref{math:LQR_feron} can also be derived as an $H_2$-norm minimization problem, where the $H_2$-norm of the transfer function from a disturbance acting on $\dot x$ to the output $z=R^{1/2}u + Q^{1/2}x$ is minimized \cite{Feron}.}%
\begin{subequations}
    \label{math:LQR_feron}
\begin{align}
    \min_{K,Y}~ &   \operatorname{tr }(QY) + \operatorname{tr }(K^\top R K Y)\\
    \text{s.t.}~ & A_K Y + Y A_K^\top + x_0x_0^\top \prec 0, \label{math:LQR_feron_stab}\\
    & Y\succ 0,
\end{align}
\end{subequations}
which admits a convex reformulation via the change of variables $Z=KY$ \cite{Feron}. 
Problem \eqref{math:LQR_feron} minimizes the LQR objective \eqref{math:feron_obj} parameterized by $K$ subject to a stability constraint \eqref{math:LQR_feron_stab}, a relaxation of \eqref{math:Lyap2}, ensuring that $K$ stabilizes $(A,B)$\footnote{The standard Lyapunov stability LMI can be recovered via the congruent transformation $P=Y^{-1}$, yielding $A_K^\top P + P A_K \prec - P x_0 x_0^\top P \preceq 0$.}. 

\begin{theorem}[\texorpdfstring{\cite[Duality]{Balakrishnan1995}}{}]
    \label{math:theo:dual} 
    Problems \eqref{math:LQR_boyd} and \eqref{math:LQR_feron} are duals of each other and strong duality holds.
\end{theorem}

This theorem confirms that the optimal solution of the LQR problem is indeed a static, linear state-feedback controller. 

\begin{remark}
    \label{remark:sdp}
    Problem \eqref{math:LQR_boyd} is a semidefinite program (SDP), and \eqref{math:LQR_feron} becomes one after the change of variables $Z=KY$ and replacing $\prec$ with $\preceq$, without altering the optimal value, allowing both to be solved efficiently.
\end{remark}

\begin{remark}
    \label{remark:I_n} 
    Since the optimal feedback $K^*$ is independent of the initial state $x_0$, one may, without loss of generality, replace $x_0 x_0^\top $ by $ I_n$ in \eqref{math:boyd_obj},\eqref{math:Lyap2} and \eqref{math:LQR_feron_stab} to obtain initial-state independent formulations (see \cite[Sec.~3.3]{bu2020clqr} for details). We adopt these formulations throughout the paper.
\end{remark}

\subsection{Policy Gradient Flow for LQR}
\label{subsec:PO_LQR} 
We present the policy gradient flow from \cite{bu2020clqr}. The set of Hurwitz stable feedback gains of system \eqref{math:lin_sys} is denoted by 
\begin{align}
   \mathcal{K} & = \{ K \in \mathbb{R}^{m \times n} \mid \operatorname{Re}(\lambda_i(A-BK)) < 0 ~ \forall i \}, \label{math:K}
\end{align}
and is open, unbounded, and path-connected \cite[Section 3]{bu2019}.
    The LQR cost function parameterized in terms of $K$ is defined as the matrix function $f_K:\mathbb{R}^{m\times n}\to \mathbb{R}$
    \begin{align}
     K\mapsto f_K= \operatorname{tr }(P_K), \label{math:cost_K} 
    \end{align} 
    where $P_K$ is defined by \eqref{math:int_P}.
The function $f_K$ exhibits favorable analytical properties \cite{bu2020clqr} for the well-definedness of its gradient flow.
Its effective domain is $\mathcal{K}$, over which $f_K$ is real analytic, coercive and hence admits compact sublevel sets. Moreover, $f_K$ is gradient dominated on each of its sublevel sets.
The policy gradient flow of \eqref{math:cost_K} is 
\begin{subequations}
    \label{math:grad_all} 
 \begin{align}
   \dot K &= -  \nabla f_K,  \quad K(0)  = K_0\in\mathcal{K},  \label{math:grad}\\
   \nabla f_K & = 2 \left( R K - B^\top P_K \right) Y_K, \label{math:grad_mod} 
\end{align}
\end{subequations}
where $P_K$ and $Y_K$ are the solutions to the Lyapunov equations \eqref{math:Lyap1} and \eqref{math:Lyap2}, respectively, with $x_0 x_0^\top$ replaced by $I_n$.
The gradient flow \eqref{math:grad_all} is well defined for $K\in\mathcal{K}$ and generates unique trajectories $K(t)$ within $\mathcal{K}$ that converge to $K^*$.

\subsection{Policy Iteration and Integral Reinforcement Learning}
We briefly present the PI and the IRL formulation for the continuous-time  LQR,  see \cite[Chapters~2--4]{vrabie2012optimal} for details.
For a stabilizing policy $u=\mu(x)$, the value function is defined as
\begin{align}
V^\mu (x_0)=\int_t^\infty x^\top Q x + \mu^\top(x) R \mu(x)\, dt <\infty  \label{math:value}
\end{align}
and satisfies the continuous-time Bellman equation
\begin{align}
   0 = \underbrace{x^\top  Q x + \mu^\top(x) R \mu(x) + \left(\nabla_x V^\mu(x)\right)^\top\dot x}_{\eqqcolon H(x,\mu(x),\nabla_x V^\mu(x))} \label{math:BE}
\end{align}
where $V^\mu(0)=0$. 
The optimal value function $V^*(x)$ solves the Hamilton-Jacobi-Bellman (HJB) equation
\begin{align}
0 = \inf_{\mu} H\bigl(x,\mu(x),\nabla_x V^*(x)\bigr),
\label{math:HJB}
\end{align}
which reduces to the CARE in the LQR setting.
\begin{algorithm}[t]
\caption{Policy Iteration \cite[Algorithm 4.1]{vrabie2012optimal}}
\label{alg:PI}
\KwIn{Stabilizing policy $\mu_0$ }
\KwOut{Optimal policy $\mu^*$}
$k=0$ \\
\While{$\mu_k \ne \mu_{k+1}$}{
  Policy evaluation: Solve Bellman equation for $\nabla_x V^{\mu_k}$
   \begin{align}
     \!\!\!   0 &  =  x^\top  Q x +  \mu_k^\top R \mu_k +  \left(\nabla_x V^{\mu_k}\right)^\top \!(Ax \!  + \! B \mu_k). \label{math:RL_eval} 
   \end{align}\\
  Policy improvement: Update the policy using 
  \begin{align}
      \mu_{k+1} &= \arg \min_{\mu} H(x,\mu(x),\nabla_x V^{\mu_k}(x)), \label{math:pol_imp}  \\
      & = - \tfrac{1}{2} R^{-1} B^\top \nabla_x V^{\mu_k}(x). \label{math:control}
   \end{align} \\
  $k=k+1$
}
\KwRet $\mu_k$
\end{algorithm}
Algorithm~\ref{alg:PI} summarizes PI for solving the nonlinear HJB equation \eqref{math:HJB}.
For the LQR problem, PI reduces to Kleinman's algorithm \cite{kleinman1968iterative}, obtained under the value function $V=x^\top Px$ and the policy $u=-Kx$. Given a stabilizing initial gain $K_0$, 
it replaces \eqref{math:RL_eval} and \eqref{math:control} in the PI by%
\begin{subequations}
    \label{math:kleinman} 
    \begin{align}
 \left(A-BK_k  \right)^{\!\top} \! \!  P_k + P_k \left(A-BK_k \right) &= -K_k^\top R K_k - Q , \label{math:P_eva}\\
 K_{k+1} & = R^{-1} B^\top P_k. \label{math:P_imp}
\end{align} 
\end{subequations}
\begin{theorem}[\texorpdfstring{\cites{kleinman1968iterative}[Theorem~3.1]{Bini2012}}{}]
    \label{theo:PI}
    Let $K_0\in\mathcal{K}$. For the iterates $K_k$ and $P_k$ of the Kleinman Algorithm \eqref{math:kleinman},
    $A-BK_k$ is Hurwitz, $P^*\preceq P_{k+1} \preceq P_k$, $\lim_{k\to\infty }K_k = K^*$ and $\lim_{k\to\infty } P_k = P^*$.
\end{theorem}

To avoid requiring knowledge of the system dynamics in the policy evaluation step, 
 \eqref{math:P_eva} can be replaced by its IRL counterpart \cite{JIANG20122699}.
Substituting the open-loop dynamics $\dot x = Ax + Bu = A_{K_k} x + B(K_k x + u)$ into the time derivative of the value function $V=x^\top P_k x$ results in
\begin{align}
 \dot V & 
 = x^\top \! \left(A^\top P _k  + P_k A  \right)x + 2 u^\top B^\top P_k x  \label{math:V_dot}\\
 & = x^\top \! \left(A_{K_k}^\top P_k + P_k A_{K_k}  \right)x + 2 (u+K_kx)^\top B^\top P_k x  \notag  \\
 &  \overset{\eqref{math:Lyap1}}{=} \! x^\top \!\left(-Q - K_k^\top R K_k  \right) x +   2 (u+K_kx)^\top B^\top P_k x. \notag
\end{align} 
Integrating this equation over a time interval $\delta>0$ yields 
\begin{align}
   & x^\top(t+\delta) P_k x(t+\delta) - x^\top(t) P_k x(t) \label{math:off_pol_eval}  \\
  &= \! - \! \int_t^{t+\delta} \! \! \! \! \! \! \! x^\top \!\left(Q \! + \! K_k^\top R K_k  \right)x\mathrm{d}\tau  \!+\! 2 \int_t^{t+\delta } \! \!\!\!\!\!\! \left(u+K_kx \right)^\top \! B^\top \!P_k  x \mathrm{d}\tau, \notag
\end{align}
which enables evaluating different policies $K$ using the same off-policy data\footnote{IRL policy evaluation was first studied in the on-policy setting \cite{VRABIE2009477}.}. Equation \eqref{math:off_pol_eval} must be evaluated over multiple time intervals to solve the associated least-squares problem for the unknown matrices $P_k$ and $B^\top P_k$.
Substituting $B^\top P_K = R K_{k+1}$ directly provides the improved policy $K_{k+1}$. 

\subsection{Riccati Flow and its Value Iteration}
\label{subsec:value_it} 
We summarize the Riccati flow and its VI as presented in \cite{BIAN2016348}.
The differential matrix Riccati equation 
\begin{align}
\dot P = A^\top P + P A - PB R^{-1} B^\top P + Q,\quad P(0)\in\mathbb{S}_+^n. \label{math:CARE_flow} 
\end{align} 
is referred to as the Riccati flow, see \cite{kuvcera1973review} for  analysis.

\begin{theorem}[\texorpdfstring{\cite[Section 2.3]{BIAN2016348}}{}]
    \label{theo:ric_flow} 
   The trajectory $P(t)$ generated by \eqref{math:CARE_flow} converges to $P^*$ for $t\to\infty$. Moreover, if $P(\tau)\succ 0$ for some $\tau\geq 0$, then $P(t)\succ 0$ for all $t\geq \tau$.
\end{theorem}

To circumvent the analytical and numerical challenges in solving \eqref{math:CARE_flow}, 
we consider a VI scheme based on stochastic approximation.
Let $\{\varepsilon_k \}_k^\infty$ be a step size sequence satisfying
\begin{align}
 \varepsilon_k >0, \quad \sum_{k=0}^\infty \varepsilon_k = \infty \quad \text{and} \quad \sum_{k=0}^\infty \varepsilon_k^2 < \infty. \label{math:vareps} 
\end{align} 
Additionally, let $\{B_q\}_{q=0}^{\infty}$  be a sequence of bounded sets with nonempty interiors such that
\begin{align}
 B_q \subseteq B_{q+1}, \quad q\in\mathbb{N}_0, \quad \lim_{q\to\infty} B_q = \mathbb{S}_+^n. \label{math:B_q} 
\end{align} 
The resulting VI is presented in Algorithm~\ref{alg:VI}.

\begin{algorithm}
\caption{Value Iteration \cite[Algorithm 1]{BIAN2016348}}
\label{alg:VI}
\KwIn{$P_0\succeq 0$ }
\KwOut{Optimal value matrix $P^*$}
$k=0,q=0$\\
\While{$P_k \ne P_{k+1}$}{
 $\tilde{P}_{k+1} = P_k + \varepsilon_k \! \left(A^\top P_k + P_k A - P_k B R^{-1} B^\top P_k + Q  \right)$ \\
 \eIf{$\tilde{P}_{k+1}\notin B_q$}{
    $P_{k+1} = P_0$, $q = q+1$
 }{
    $P_{k+1} = \tilde{P}_{k+1}$
 }
  $k=k+1$
}
\KwRet $P_k$
\end{algorithm}

\begin{theorem}[\texorpdfstring{\cite[Theorem 3.3]{BIAN2016348}}{}]
    \label{theo:VI} 
    The sequence $\{P_k\}_{k=0}^\infty$ generated by Algorithm~\ref{alg:VI} converges to $P^*$ for $k\to\infty$.
\end{theorem}

\subsection{Data-based Closed-loop Representation}
\label{subsec:persis} 

We briefly recall the data-based closed-loop representation in continuous time from \cite{de2019formulas}. 
Consider a sequence of $T$ measurements of state, input, and state derivative trajectories of system \eqref{math:lin_sys}, sampled at interval $\Delta>0$,
\begin{align}
   X &= \begin{bmatrix} x(0) & x(\Delta) & \dots & x((T-1)\Delta) \end{bmatrix}\in\mathbb{R}^{n\times T},\\
   U & =\begin{bmatrix} u(0) & u(\Delta) & \dots & u((T-1)\Delta)  \end{bmatrix}\in\mathbb{R}^{m\times T}, \\
   \dot{X} & = \begin{bmatrix} \dot x(0) & \dot x(\Delta) & \dots & \dot x((T-1)\Delta) \end{bmatrix}\in\mathbb{R}^{n\times T}.
   \end{align}
   that satisfy%
   \begin{align}
    \dot{X} & = A X + BU= \begin{bmatrix} B & A \end{bmatrix} \begin{bmatrix} U \\ X \end{bmatrix}. \label{math:ex_sys} 
   \end{align}

\begin{theorem}[\texorpdfstring{ \cite[Theorem 2 \& Remark 2]{de2019formulas}}{}]
   \label{theo:1}
   Let $ \operatorname{rank} \left[\begin{smallmatrix}U \\ X\end{smallmatrix}\right]  = n+m$ hold. Then, the closed-loop system $\dot x = A_K x$ can equivalently be represented using data as%
   \begin{subequations}
    \label{math:data_sys}
    \begin{align}
      \dot x & = \dot{X} G_K x, \\
      \begin{bmatrix} -K \\ I_n \end{bmatrix} & = \begin{bmatrix} U \\ X \end{bmatrix} G_K, \label{math:eq1}
   \end{align}
   \end{subequations}
   where $G_K\in\mathbb{R}^{T \times n}$.
\end{theorem}

\section{Data-driven Parameterizations and Their Analytical Properties}
\label{sec:param} 
In this section, we present both parameterizations,  derive their analytical properties, and clarify their relationship.
\subsection{Closed-loop Parameterization}
Inspired by \cite{Lopez2023}, we slightly modify the data-based CL parameterization of \cite{de2019formulas} (see Subsection \ref{subsec:persis}) to remove the requirement for state-derivative measurements, thereby avoiding noise amplification caused by numerical differentiation.
Integrating \eqref{math:lin_sys} over a time interval $\delta>0$ results in
\begin{align}
    x(t+\delta) - x(t) & = A \int_t^{t+\delta}  x(\tau) \mathrm{d}\tau + B \int_t^{t+\delta}  u(\tau) \mathrm{d}\tau. \label{math:sys_int}
\end{align}
Repeating this integration over $T$ distinct time intervals yields
\begin{align}
     \bar{X} =  A \tilde{X} + B \tilde{U} =  \begin{bmatrix}B & A \end{bmatrix} \begin{bmatrix} \tilde{U}\\ \tilde{X} \end{bmatrix},\label{math:ex_sys_int} 
\end{align}
where 
\begin{align}
    \bar{X} &  \! =  \! \begin{bmatrix} x(t_1 \! + \!\delta) \!- \!x(t_1) & \!  \! \!  \dots  \! \! \!  & x(t_T  \! +  \! \delta) \!- \! x(t_T)\end{bmatrix}  \!\in  \!\mathbb{R}^{n\times T},\\
    \tilde{U} &  \!= \! \begin{bmatrix} \int_{t_1}^{t_1+\delta }u(\tau)\mathrm{d}\tau & \dots &\int_{t_T}^{t_T+\delta }u(\tau)\mathrm{d}\tau \end{bmatrix}\in\mathbb{R}^{m\times T},\\
     \tilde{X} & \! =  \! \begin{bmatrix} \int_{t_1}^{t_1+\delta }x(\tau)\mathrm{d}\tau & \dots &\int_{t_T}^{t_T+\delta }x(\tau)\mathrm{d}\tau \end{bmatrix}\in\mathbb{R}^{n\times T}.
\end{align}
The following assumption ensures that the data is informative to characterize the system \eqref{math:lin_sys} and holds throughout the paper.
\begin{assumption}
    \label{ass:2} 
    The matrix $\begin{bmatrix}\tilde{U} \\ \tilde{X}\end{bmatrix}$ has rank $n+m$.
\end{assumption}

By \cite[Lemma~4]{Lopez2023}, Assumption~\ref{ass:2} holds under piecewise-constant inputs that are persistently exciting of order $n+1$, which requires $T\geq (m+1)n+m$. Next, we adapt Theorem~\ref{theo:1} to the integrated dynamics \eqref{math:sys_int}.

\begin{theorem}
    \label{theo:closed} 
     The closed-loop system $\dot x = A_K x$ can equivalently be represented as%
\begin{subequations}
    \label{math:sys_closed} 
    \begin{align}
 \dot x &= \bar{X} G x,  \\
 \begin{bmatrix}-K \\ I_n \end{bmatrix}& = \begin{bmatrix}\tilde{U} \\ \tilde{X} \end{bmatrix} G, \label{math:data_c}
\end{align} 
\end{subequations}
where $G\in\mathbb{R}^{T\times n}$.
\end{theorem}

\begin{proof}
    The proof is similar to that of Theorem~\ref{theo:1} by using \eqref{math:ex_sys_int} instead of \eqref{math:ex_sys}.
\end{proof}

\begin{remark}
    The matrices $\bar{X}, \tilde{U}, \tilde{X}$ may be replaced by $\dot{X},U,X$, respectively, without altering the subsequent results.
\end{remark}

Let $\mathcal{G}$ denote the set of all $G\in\mathbb{R}^{T\times n}$ satisfying \eqref{math:data_c} and rendering $\bar{X} G$ Hurwitz, i.e., 
\begin{align}
 \mathcal{G} &\coloneqq\{G\in\mathbb{R}^{T\times n} \mid  I_n = \tilde X G, \operatorname{Re}(\lambda_i(\bar{X} G ))<0~ \forall i \}  \label{math:G} \\
 & ~  = \{G\in\mathbb{R}^{T\times n}   \mid   \exists K\in\mathcal{K} \text{ s.t. }  \begin{bmatrix} - K \\ I_n \end{bmatrix} = \begin{bmatrix} \tilde{U} \\ \tilde{X} \end{bmatrix} G\}. \notag
\end{align} 

Analogous to \eqref{math:cost_K}, we define a data-driven LQR cost based on the equivalence representation of Theorem~\ref{theo:closed}. 
\begin{definition}
    \label{def:LQR_1} 
    The LQR cost under  the CL parameterization is defined as the matrix function 
    \begin{align}
    f_G:\mathcal{G}\to\mathbb{R}, ~ G\mapsto f_G = \operatorname{tr}\left(P_G  \right), \label{math:f_G}  
    \end{align} 
    where $P_G$ is the solution of the Lyapunov equation
       \begin{align}
     0 & = (\bar{X}G  )^\top P_G + P_G (\bar{X} G ) + Q + (\tilde{U} G  )^\top R (\tilde{U} G ). \label{math:Pol_eval} 
    \end{align} 
\end{definition}

We consider the minimization of the LQR cost \eqref{math:f_G}, i.e., 
\begin{align}
 \min_{G\in\mathcal{G}} f_G, \label{math:data_op1} 
\end{align} 
such that any optimal solution $G^*$ yields the optimal feedback $K^*= - \tilde{U} G^*$. Several solution approaches to \eqref{math:data_op1} are developed in the subsequent sections under different settings. 
Next, we state the analytical properties of this parameterization. 

\begin{lemma}
    \label{lem:G_K} 
    Consider the system \eqref{math:sys_closed}.
    \begin{enumerate}
        \item \label{en_1} Let $T>m+n$. Then, for any $K\in\mathbb{R}^{m\times n}$, the solution $G\in\mathbb{R}^{T\times n}$ to \eqref{math:data_c} is not unique.
        \item \label{en_2}The general solution to \eqref{math:data_c} is $G = G^p + N$, where 
        \begin{align}
         G^p =    \begin{bmatrix} \tilde{U} \\ \tilde{X} \end{bmatrix}^\dagger \begin{bmatrix} -K \\ I_n \end{bmatrix} \quad \text{and} ~~ \begin{bmatrix}\tilde{U} \\ \tilde{X}\end{bmatrix} N = 0.
         \label{math:gen_sol} 
        \end{align} 
        \item \label{en_3} Each solution $G$ to \eqref{math:data_c} renders $K=-\tilde{U}G$ and $A_K=\bar{X}G$ uniquely.
      \end{enumerate}
\end{lemma}

\begin{proof}  
    Assumption~\ref{ass:2} and $T>m+n$ imply that $\left[\begin{smallmatrix}\tilde{U} \\ \tilde{X}\end{smallmatrix}\right]$ has a nontrivial null space. By the rank-nullity theorem, $\dim \ker\bigl(\left[\begin{smallmatrix}\tilde{U} \\ \tilde{X}\end{smallmatrix}\right]\bigr) = T - m - n > 0$, implying claim \ref{en_1}. 
    Since $\left[\begin{smallmatrix}\tilde{U} \\ \tilde{X}\end{smallmatrix}\right]$ has full row rank, its pseudoinverse is $\left[\begin{smallmatrix}\tilde{U} \\ \tilde{X}\end{smallmatrix}\right]^\top (\left[\begin{smallmatrix}\tilde{U} \\ \tilde{X}\end{smallmatrix}\right] \left[\begin{smallmatrix}\tilde{U} \\ \tilde{X}\end{smallmatrix}\right]^\top )^{-1}$. Substituting $G^p$ into \eqref{math:data_c} verifies that $G^p$ is a particular solution. 
    Since $ \operatorname{ker}(\left[\begin{smallmatrix}\tilde{U} \\ \tilde{X}\end{smallmatrix}\right])$ is nontrivial, the general solution is the sum of $G^p$ and an arbitrary $N$ in this null space, proving claim \ref{en_2}. Uniqueness in claim \ref{en_3} follows since $\tilde{U}N=0$ and, by \eqref{math:ex_sys_int}, $\bar{X}N = A\tilde{X}N + B\tilde{U}N = 0$.
\end{proof}

Note that Lemma~\ref{lem:G_K} implies that distinct matrices $G$ satisfying \eqref{math:gen_sol} yield an identical closed-loop matrix $A_K = \bar{X} G$. 

\begin{remark}
   The particular solution $G^p$ in \eqref{math:gen_sol} is the minimal Frobenius norm solution to \eqref{math:data_c}, i.e.,%
   \begin{align}
      G^p = \arg \min_G \Vert G \Vert_F^2 \quad \text{s.t.} \quad \eqref{math:data_c}.
   \end{align}
\end{remark}

\begin{lemma}
    \label{lemma:pos_s} 
Let $G \in \mathcal{G}$. Then the solution $P_G$ to \eqref{math:Pol_eval} is unique and positive semidefinite. If $Q \succ 0$ or $(A, \sqrt{Q})$ is observable, then $P_G \succ 0$. Furthermore, the image of the function $f_G$ \eqref{math:f_G} satisfies $f_\mathcal{G} = \{f_G \mid G\in\mathcal{G}\}\subseteq \mathbb{R}_{\geq 0}$.
\end{lemma}

The proof is provided in Appendix \ref{app:proof1}.
   For  $G\in\mathcal{G}$, the solution to 
   \eqref{math:Pol_eval} can be obtained via vectorization 
   \begin{align}
      \begin{split}
         \label{math:vec}
      \operatorname{vec}(P_G ) & = - \bigl(I_n \otimes (\bar{X} G  )^\top +   (\bar{X} G )^\top \otimes I_n  \bigr)^{-1} \\
      & \quad ~\operatorname{vec}\left( Q + (\tilde U G)^\top R  \tilde U G \right).
      \end{split}
   \end{align}
However, solving Lyapunov equations via vectorization can be computationally expensive and numerically ill-conditioned. Efficient alternatives are the Bartels--Stewart \cite{bartels1972solution} and the Hessenberg--Schur \cite{Hessenberg1979} algorithm. 
Motivated by \cite[Lemma 2.1]{dorfler2022}, we characterize the optimal solutions to \eqref{math:data_op1}.
\begin{lemma}
    \label{lem:G_star}
    The optimal solution $G^*$ to the problem \eqref{math:data_op1} is not unique for $T>m+n$ and is contained in
    \begin{align}
     \mathcal{G}^* = \{ G\in\mathbb{R}^{T\times n}\mid G = \begin{bmatrix} \tilde{U} \\ \tilde{X} \end{bmatrix}^\dagger \begin{bmatrix} -K^* \\ I_n \end{bmatrix} + N  \} \subset \mathcal{G},
    \end{align} 
    where  $\left[\begin{smallmatrix}\tilde{U} \\ \tilde{X}\end{smallmatrix}\right] N = 0$ and $K^*$ is the optimal LQR gain.
\end{lemma}
\begin{proof}
    For any $G\in\mathcal{G}^*$, Lemma~\ref{lem:G_K} gives $-\tilde{U}G=K^*$ and $\bar{X}G=A_{K^*}$, so \eqref{math:Pol_eval} has the unique solution $P_{K^*}$ and $f_G=f_{K^*}$. Hence, every $G\in\mathcal{G}^*$ is optimal. Conversely, let $G^*$ solve \eqref{math:data_op1}. By Theorem~\ref{theo:closed}, the induced gain $\hat{K}=-\tilde{U}G^*$ satisfies $f_{\hat{K}}=f_{G^*}=f_{K^*}$. Since the optimal LQR gain is unique, $\hat{K}=K^*$, and Lemma~\ref{lem:G_K} yields $G^*\in\mathcal{G}^*$. Nonuniqueness follows from Lemma~\ref{lem:G_K}, since $T>m+n$ admits $N\neq 0$.
\end{proof}

Next, we show that $\mathcal{G}$ enjoys some favorable properties.
\begin{lemma}
    \label{lem:set} 
   The set $\mathcal{G}$ is path-connected, unbounded, and relatively open in the affine subspace%
   \begin{align}
      \mathcal{S}\coloneqq\{G \in\mathbb{R}^{T\times n}: \tilde{X} G = I_n\}. 
   \end{align}
\end{lemma}
The proof is given in Appendix \ref{app:proof2}.

\begin{lemma}
    \label{lemma:f_G_ana} 
    The LQR cost function \eqref{math:f_G} is real analytic over $\mathcal{G}$, i.e., $f_G \in C^\omega(\mathcal{G})$.
\end{lemma}

\begin{proof}
By Cramer's rule, the map $G\mapsto P_G$ is rational in the entries of $G$ (see \eqref{math:vec}) and therefore belongs to $C^\omega(\mathcal{G})$. Since the trace is linear, the composition $G\mapsto \operatorname{tr}(P_G)$ is also in $C^\omega(\mathcal{G})$. Hence, $f_G\in C^\omega(\mathcal{G})$.
\end{proof}

Since $f_G$ is real analytic over $\mathcal{G}$, it is infinitely differentiable with smooth derivatives, ensuring the well-posedness of the associated gradient flow. We now characterize the boundary $\partial \mathcal{G}$ of $\mathcal{G}$ relative to $\mathcal{S}$ as
\begin{align}
 \partial \mathcal{G}  \subseteq  \! \{G \in\mathbb{R}^{T\times n }\! \mid  \! I_n 
   = \tilde X G,  \max_i \operatorname{Re}\bigl(\lambda_i(\bar{X} G)\bigr)=0 \}. \label{math:G_boundary}
\end{align}
The inclusion in \eqref{math:G_boundary} may be strict, since a matrix $G$ satisfying $I_n=\tilde{X}G$ for which $\bar{X}G$ has only eigenvalues with nonpositive real part, with at least one eigenvalue on the imaginary axis, need not be the limit of a sequence in $\mathcal{G}$.\footnote{An analogous phenomenon occurs for $\mathcal{K}$ \eqref{math:K}: for $m>1$, not every $K$ for which $A_K$ has only eigenvalues with nonpositive real part and at least one on the imaginary axis lies on $\partial\mathcal{K}$, see \cite[Prop.~3.5]{bu2019}.}

\begin{lemma}
    \label{Lemma:noncoercive} 
   For $T>m+n$, the LQR cost \eqref{math:f_G} is not coercive over $\mathcal{G}$, i.e., $f_G\to\infty$ as $G\to\partial\mathcal{G}$, while $f_G$ remains bounded along some sequences in $\mathcal{G}$ with $\Vert G \Vert_2 \to \infty$.
\end{lemma}
The proof appears in Appendix \ref{app:proof3}.

\begin{remark}
    Unlike the model-based LQR cost \eqref{math:cost_K}, the LQR cost $f_G$ under the CL parameterization is noncoercive for $T>m+n$, and its sublevel sets are therefore noncompact.
\end{remark}

\subsection{Integral Reinforcement Learning Parameterization}
The parameterization and some results of this subsection follow those given in \cite{JIANG20122699} and \cite{BIAN2016348}. 
The IRL parameterization is based on \eqref{math:off_pol_eval}, which can equivalently be written as 
\begin{align}
    \operatorname{tr}\Bigl(P \bigl(\underbrace{x(t+\delta ) x(t+\delta )^\top - x(t)x(t)^\top}_{\eqqcolon r_{\Delta x}(t)} \bigr) \Bigr) \notag \hspace{2cm}  \\
    \quad  = - \operatorname{tr}\Bigl( \left(Q + K^\top R K  \right)\underbrace{\int_t^{t+\delta} x x^\top \mathrm{d}\tau}_{\eqqcolon r_{xx}(t)} \Bigr)  \hspace{2cm} \label{math:int_lyap} \\
    \quad ~\! \! \! + 2 \operatorname{tr}\biggl(   B^\top P  \Bigl(\underbrace{\int_t^{t+\delta} xu^\top \mathrm{d}\tau}_{\eqqcolon r_{xu}(t)} +  \int_t^{t+\delta} x x^\top \mathrm{d}\tau K^\top \Bigr) \biggr) \notag
\end{align} 
or
\begin{multline}
 -  \operatorname{vec}\left(r_{x x}(t) \right)^\top \! \operatorname{vec}\left(Q\! + \!K^\top \! R K  \right)  =  \operatorname{vec}\left(r_{\Delta x}(t) \right)^\top \! \operatorname{vec}\left(P  \right) \\ - 2  \operatorname{vec}\left( (r_{xu}(t)+  r_{xx}(t) K^\top )^\top \right)^\top \operatorname{vec}\left(B^\top P  \right), 
\end{multline} 
where $r_{\Delta x}(t), r_{xx}(t)$ and $r_{xu}(t)$ are measured under the system \eqref{math:lin_sys}.
By defining 
\begin{align}
 \Gamma^{\Delta x} &= \begin{bmatrix} \operatorname{vec}\left(r_{\Delta x}(t_1) \right)  &  \! \! \! \dots \! \! \!  & \operatorname{vec}\left( r_{\Delta x}(t_T)\right) \end{bmatrix}^{\!\top} \! \in\mathbb{R}^{T \times n^2} ,\\
\Gamma^{xx}& = \begin{bmatrix} \operatorname{vec}\left(r_{x x}(t_1) \right)  &  \! \! \! \dots  \! \! \!  & \operatorname{vec}\left(r_{xx}(t_T) \right)\end{bmatrix}^\top \! \in\mathbb{R}^{T \times n^2} ,  \\
\Gamma^{ux} & = \begin{bmatrix} \operatorname{vec}\bigl(r_{x u}^{ \top}(t_1) \bigr) &   \! \! \! \dots  \! \! \!  & \operatorname{vec}\bigl(r_{x u}^{ \top}(t_T) \bigr)\end{bmatrix}^\top \in\mathbb{R}^{T \times mn}, 
    \end{align} 
the system of linear equations of \eqref{math:int_lyap} for $T$ samples is 
\begin{align}
 \underbrace{- \Gamma^{xx} \operatorname{vec}\left(Q+K^\top R K  \right)}_{\eqqcolon b\in\mathbb{R}^T } \hspace{4cm} \notag
 \\ 
 = \underbrace{\begin{bmatrix}\Gamma^{\Delta x} & -2E(K)\end{bmatrix}}_{\eqqcolon \bar{\Phi}\in\mathbb{R}^{T\times (n^2+mn)}} \begin{bmatrix} \operatorname{vec}\left(P  \right)\\ \operatorname{vec}\left(B^\top P  \right)\end{bmatrix} \quad ~\!\! \label{math:LSE_0} 
 \\ = \underbrace{\begin{bmatrix}\Gamma^{\Delta x}D & -2E(K)\end{bmatrix}}_{\eqqcolon \Phi\in\mathbb{R}^{T\times (n(n+1)/2+mn)}} \underbrace{\begin{bmatrix} \operatorname{vech}\left(P  \right)\\ \operatorname{vec}\left(B^\top P\right)\end{bmatrix}}_{\eqqcolon\theta}, \label{math:LSE} 
\end{align}
where $E(K)\coloneqq \Gamma^{ux}+\Gamma^{xx}\left(I_n \otimes K^\top \right)$ and the duplication matrix $D\in\mathbb{R}^{n^2 \times \frac{n(n+1)}{2}}$ satisfies $\operatorname{vec}\left(P  \right) = D \operatorname{vech}\left(P \right)$. 
Equation \eqref{math:LSE} determines $P$ and $B^\top P$ for a given $K\in\mathcal{K}$ from the data matrices $\Gamma^{xx},\Gamma^{ux}$, and $ \Gamma^{\Delta x}$, thereby providing an indirect representation of the underlying system.

As in the CL parameterization, sufficiently rich data is required to represent the system behavior in \eqref{math:LSE}. 
\begin{lemma}[\texorpdfstring{\cite[Lemma 6]{JIANG20122699}}{}]
    \label{lemma:rank} 
    If there exists an integer $T_0>0$, such that, for all $T\geq T_0$, 
    \begin{align} \label{math:rank_cond} 
     \operatorname{rank} \begin{bmatrix} \Gamma^{xx} & \Gamma^{ux} \end{bmatrix} = \frac{n(n+1)}{2}+mn 
    \end{align} hold, then $\Phi$ has full column rank.
\end{lemma}

A necessary condition for \eqref{math:rank_cond} is $T\geq T_0\geq \frac{n(n+1)}{2}+mn$. 
We  impose the following standing assumption.
\begin{assumption}
    \label{ass:rank_2} 
    Condition \eqref{math:rank_cond} holds. 
\end{assumption}

Under Assumption~\ref{ass:rank_2}, the least-squares problem $     \min_\theta \Vert b - \Phi\theta \Vert_2$ of \eqref{math:LSE} admits the  minimal Euclidean norm solution
\begin{align}
 \label{math:theta} 
 \hat \theta &= \Phi^\dagger b = \left(\Phi^\top  \Phi \right)^{-1}  \Phi^\top b =  \begin{bmatrix} \operatorname{vech}(\hat P)\\ \operatorname{vec}(B^\top \hat P  ) \end{bmatrix}
\end{align} 
which satisfies  $\Vert b-\Phi\hat \theta \Vert_2 =0$.
Note that the consistency of \eqref{math:LSE} relies on noise-free measurements $\Gamma^{\Delta x}, \Gamma^{xx}$ and $\Gamma^{ux}$ satisfying the dynamics \eqref{math:lin_sys} and Assumption~\ref{ass:rank_2}.

\begin{lemma}
    \label{lemma:pol_imp} 
    For any $K\in\mathcal{K}$, the matrix $\hat P$ obtained from \eqref{math:theta} coincides with the solution of the Lyapunov equation~\eqref{math:Lyap1}. 
\end{lemma}
\begin{proof}
  Since $K\in\mathcal{K}$, \eqref{math:Lyap1} admits a unique solution $P_0\succeq 0$.
    By definition of $b$ and $\Phi$, this solution satisfies
    \begin{align}
      \Gamma^{xx} \operatorname{vec}\left( A_K^\top P_0 + P_0 A_K  \right) = b = \Phi \begin{bmatrix} \operatorname{vech}\left(P_0  \right)\\ \operatorname{vec}\left(B^\top P_0\right)\end{bmatrix}. \label{math:LSE_2} 
    \end{align} 
    Hence, $\theta_0 \coloneqq \begin{bmatrix} \operatorname{vech}(P_0)^\top & \operatorname{vec}(B^\top P_0)^\top \end{bmatrix}^\top $  is a feasible solution to \eqref{math:LSE_2}. Under Assumption~\ref{ass:rank_2}, \eqref{math:LSE_2} has a unique least-squares solution \eqref{math:theta}. Therefore, $\hat{\theta}=\theta_0$, implying  $\hat{P}\!=\!P_0$.
\end{proof}

\begin{remark}
    \label{remark_1} 
     Lemma~\ref{lemma:pol_imp} implies that \eqref{math:LSE}  enables a data-driven policy evaluation, analogous to \eqref{math:Pol_eval}. In particular, \eqref{math:theta} induces a mapping $g:\mathcal{K}\to S^n_{+}$ defined by\footnote{The inverse of the vectorization operator $\operatorname{vec}^{-1}_{k\times l}: \mathbb{R}^{k l}\to \mathbb{R}^{k\times l}$ is defined as
        $A\mapsto  \operatorname{vec}^{-1}_{k\times l}(A) = \left( \operatorname{vec}^\top(I_l ) \otimes I_k\right)(I_l \otimes A)$.}
    \begin{align}
        \label{math:g} 
      K\mapsto \hat{P}_K & \! = \! \operatorname{vec}^{-1}_{n \times n} \left( D \begin{bmatrix} I_{p} & 0_{p,mn}\end{bmatrix} \hat{\theta}_K \right) \!,~  p=\tfrac{n(n+1)}{2}
    \end{align} 
    where $\hat{\theta}_K$ denotes the least-squares estimate \eqref{math:theta} parameterized by the feedback gain $K$. The mapping \eqref{math:g} assigns to any $K\in\mathcal{K}$ a unique matrix $\hat P_K\succeq 0 $ satisfying \eqref{math:Lyap1}. 
\end{remark}

\begin{definition}
    \label{def:LQR_2} 
    The LQR cost under the IRL parameterization is defined as the matrix function 
    \begin{align}
        \label{math:f_IRL} 
      \hat f_K : \mathcal{K}\to \mathbb{R},~K \mapsto \hat f_K  = \operatorname{tr} ( \hat{P}_K ),
    \end{align}
    where $\hat{P}_K$ satisfies \eqref{math:LSE} and is obtained by \eqref{math:g}.
\end{definition}

Note that $\hat{f}_K$ denotes the LQR cost under IRL parameterization, where $\hat{P}_K$ is obtained from data, whereas $f_K$ corresponds to the model-based formulation in \eqref{math:cost_K}.
Analogous to \eqref{math:data_op1}, we consider minimizing the data-driven LQR cost \eqref{math:f_IRL}, i.e.
    \begin{align}
        \label{math:f_hat} 
     \min_{K\in\mathcal{K}} \hat{f}_K.
    \end{align} 

\begin{remark}
    \label{remark:same_P}
    By construction, the functions $f_K$ \eqref{math:cost_K} and $\hat{f}_K$ \eqref{math:f_IRL} coincide on the domain $\mathcal{K}$. Consequently, $\hat{f}_K$ inherits the analytical properties of $f_K$, in particular, $\hat{f}_K$ is real analytic, coercive and hence admits compact sublevel sets over $\mathcal{K}$.
\end{remark}

For the Riccati flow and the VI under IRL parameterization, a different set of linear equations is required. 
Following \cite{BIAN2016348}, the integrated version of \eqref{math:V_dot} is
\begin{align}
 \operatorname{tr}\left(P r_{\Delta x}(t)  \right) 
 & =  \operatorname{tr} \left(H r_{xx}(t)  \right) + 2 \operatorname{tr}\left(R K^+ r_{xu}(t) \right) 
\end{align} 
where $H = A^\top P + P A $ and $K^+=R^{-1}B^\top P$. This yields
\begin{align}
    \label{math:LSE2} 
 \underbrace{\Gamma^{\Delta x} \operatorname{vec}\left(P  \right)}_{\eqqcolon \tilde{b}}& = \underbrace{ \begin{bmatrix} \Gamma^{xx}D & 2 \Gamma^{ux} (I_n \otimes R) \end{bmatrix}}_{ \eqqcolon \tilde{\Phi}} \underbrace{\begin{bmatrix} \operatorname{vech}\left(H  \right)\\ \operatorname{vec}\left(K^+ \right)\end{bmatrix}}_{\tilde{\theta}},
\end{align} 
where $D\in\mathbb{R}^{n^2 \times \frac{n(n+1)}{2}}$ satisfies $\operatorname{vec}\left(H  \right) = D \operatorname{vech}\left( H \right)$. 

Equation \eqref{math:LSE2} has a unique solution $\hat{\tilde{\theta}}=\tilde{\Phi}^\dagger \tilde{b}$ if $\tilde{\Phi}$ has full column rank, leading to the following standing assumption.
\begin{assumption}
    \label{ass:rank_3} 
    The matrix  $\tilde \Phi$ has rank $\tfrac{n(n+1 )}{2 }+mn$.
\end{assumption}

\begin{remark}
    \label{remark:IRL} 
    Although the fundamental equations \eqref{math:LSE} and \eqref{math:LSE2} of the IRL parameterization are both constructed from a data-driven value function and rely on the same data matrices, they enable distinct parameter recovery.  In particular, \eqref{math:LSE} provides access to $P$ and $B^\top P$ for a given gain $K$, whereas \eqref{math:LSE2} yields $A^\top P + P A$ and $B^\top P$ for a given matrix $P$.
\end{remark}

The IRL parameterization may seem restricted to the LQR setting since  policy evaluation is embedded in \eqref{math:LSE} (see Remark~\ref{remark_1}), but it is not limited to this setting.
Integrating 
\begin{align}
  \frac{d }{dt} \left(x x^\top  \right) = A x x^\top + B u x^\top + x x^\top A^\top + x u^\top B^\top 
\end{align} 
over a time interval $\delta >0$ yields
    \begin{align}
      r_{\Delta x }(t) = A r_{xx}(t) + B r_{xu}^\top(t) + r_{xx}(t) A^\top \! + r_{xu}(t) B^\top \! . \label{math:alt_1} 
\end{align} 
Aggregating the vectorized form of \eqref{math:alt_1} over multiple time intervals gives,  in a form structurally similar to~\eqref{math:ex_sys_int},%
\begin{align}
(\Gamma^{\Delta x})^\top \! = \mathcal{A}(\Gamma^{xx})^\top \! + \mathcal{B}(\Gamma^{xu} )^\top = \begin{bmatrix} \mathcal{B}& \mathcal{A} \end{bmatrix} \begin{bmatrix} (\Gamma^{xu})^{\top} \\ (\Gamma^{xx})^\top \end{bmatrix}, \label{math:ext_ss} 
\end{align} 
where $ \Gamma^{xu} = \Gamma^{ux} C_{n,m}$,
 $\mathcal{A}  = (I_n \otimes A) + (A \otimes I_n) \in \mathbb{R}^{n^2\times n^2}$, and $
 \mathcal{B} = (I_n \otimes B) C_{n,m} + (B\otimes I_n )\in\mathbb{R}^{n^2 \times mn}$. 
 Since $\mathcal{A}$ and $\mathcal{B}$ depend affinely on $A$ and $B$, respectively, the system $(A,B)$ can be identified via a least-squares method from \eqref{math:ext_ss} if Assumption~\ref{ass:rank_2} holds.\footnote{Equation \eqref{math:ext_ss} can be written in linear regression form as $\operatorname{vec}(( \Gamma^{\Delta x})^\top) = \hat \Phi \bigl[ \operatorname{vec}\left(A \right)^\top \operatorname{vec}\left( B \right)^\top \bigr]^\top$\!\!, where $\hat{\Phi}$ is constructed from $\Gamma^{xx}$ and $\Gamma^{xu}$. } This demonstrates that the IRL parameterization enables data-driven formulations for general controller designs.

\subsection{Comparison of Both Parameterizations}
\label{subsec:comp} 
For valid parameterization, the rank condition of Assumption~\ref{ass:2} must hold for the CL parameterization, whereas Assumptions~\ref{ass:rank_2} or~\ref{ass:rank_3} are required for the IRL parameterization.
For the CL parameterization, Assumption~\ref{ass:2} is ensured by a persistency of excitation condition, which requires for controllable $(A,B)$ that $u$ is piecewise constant and that the input samples are persistently exciting of order $n+1$ \cites[Remark~2]{de2019formulas}[Lemma~1]{lopez_continuous-time_2022}[Lemma~4]{Lopez2023}. This enables the rank condition to be enforced a priori through input design and forces $T\geq (m+1)n+m$.
In contrast, an equivalent condition has not yet been established for the IRL parameterization, and the rank condition is thus verified only after data collection. 
Moreover, the resulting sample complexity differs between the parameterizations. The minimum number of samples required to satisfy the respective rank conditions is $T_{\mathrm{CL}}=m+n$ for the CL parameterization and $T_{\mathrm{IRL}}=\tfrac{n(n+1 )}{2 }+mn$ for the IRL parameterization, respectively. 
Their difference is
 $T_{\mathrm{IRL}} - T_{\mathrm{CL}} = (n-1)\bigl( m + \frac{n }{2}\bigr)\geq 0$,
implying $T_{\mathrm{CL}} = T_{\mathrm{IRL}} $ for $n=1$ and $T_{\mathrm{IRL}}> T_{\mathrm{CL}}$ otherwise.
Therefore, the IRL parameterization incurs higher required sample complexity, scaling quadratically with the state dimension, whereas the CL parameterization scales linearly.

The parameterizations further differ in their data structure. The CL parameterization is based on integrals of terms that are linear in $x$ and $u$ (see \eqref{math:sys_int}), whereas the IRL parameterization relies on integrals of quadratic terms in $x$ and $u$ (see \eqref{math:int_lyap}). Even after removing redundant columns of $\Gamma^{\Delta x}$ and $\Gamma^{xx}$, the IRL data matrices remain higher dimensional, implying that IRL operates in an expanded space while avoiding the equality constraint  $I_n = \tilde{X} G$. 
Moreover, integrals of linear terms average measurement noise, whereas integrals of quadratic terms are biased by it, since the square of zero-mean noise has a positive mean, and increase implementation effort due to the larger number of required integrators.
From an identification perspective, \eqref{math:ext_ss} has a structure similar to \eqref{math:ex_sys_int}. Exploiting the structure of $\mathcal{A}$ and $\mathcal{B}$ keeps the number of unknown parameters to $n^2+mn$, at the expense of a larger sample requirement under the IRL parameterization compared to standard least-squares identification under the CL parameterization.

Structurally, the CL parameterization directly represents the closed-loop matrix via $A_K = \bar{X} G $, whereas the IRL parameterization identifies the value-function quantities $B^\top P$, and $P$ or $A^\top P + P A$ (see Remark~\ref{remark:IRL}), 
which is  advantageous for formulating PI and VI schemes. The decision variables also differ in dimension. 
The CL parameterization optimizes over $G\in\mathbb{R}^{T\times n}$, whereas the IRL parameterization involves $n^2+mn$ parameters. Furthermore, the optimal solution is unique under the IRL parameterization and satisfies $K^* = R^{-1}(B^\top P^*)$, whereas the CL parameterization solution $G^*$ is nonunique due to the null space of $\begin{bmatrix}\tilde{U}^\top  & \tilde{X}^\top\end{bmatrix}^\top$. This ambiguity can be resolved via the sample covariance parameterization proposed in \cite{Zhao2025,ZhaoCovariance}.
Overall, the CL parameterization provides a direct and structurally transparent formulation with lower sample complexity, whereas the IRL parameterization trades these properties for a value-function-based formulation that is, as shown in the following sections, naturally suited to PI and VI schemes.

\section{Policy Iteration Algorithms}
\label{sec:PI} 
In this section, we present data-driven PI algorithms for the two parameterizations introduced in the previous section. 
\subsection{Closed-loop Parameterization}
Under the CL parameterization, policy evaluation is equivalent to solving 
\eqref{math:Pol_eval}, which admits a unique solution $P_G\succeq 0$ for each $G\in\mathcal{G}$, as established in Lemma~\ref{lemma:pos_s}.
Analogous to \eqref{math:pol_imp}, we formulate a policy improvement step.
\begin{proposition}
    \label{prop:improved_G} 
    Let $P\succeq 0$. An improved policy is $K=-\tilde{U}G$, where $G$ is the solution to%
    \begin{subequations}
        \label{math:opt_pol_eval} 
    \begin{align}
        \min_G  & \operatorname{tr} \bigl( Q + G^\top \tilde{U}^\top R \tilde{U} G + P \bar{X} G + (\bar{X}G)^\top P\bigr) \label{math:obj_G}  \\
        \text{s.t.}~ & I_n  = \tilde{X}G. 
    \end{align} 
    \end{subequations} 
\end{proposition}
\begin{proof}
    Substituting the CL parameterization \eqref{math:sys_closed}, the policy $u=\tilde{U}G x$, and the value function $V(x)=x^\top P x$ into the Hamiltonian \eqref{math:BE} gives
    \begin{align}
     H(x,G) & =  x^\top \bigl(Q \! +\!  G^\top \tilde{U}^\top R \tilde{U} G \!+\! P \bar{X} G \!+\! (\bar{X}G)^\top \! P \bigr) x  \\
         & = \operatorname{tr}\bigl( ( Q \! +\!  G^\top \tilde{U}^\top R \tilde{U} G \!+\! P \bar{X} G \!+\! (\bar{X}G)^\top P ) (x x^\top )  \bigr) \notag.
    \end{align} 
    Policy improvement minimizes $H$ for all $x$. Equivalently, it suffices to minimize $\sum_{i=1}^{n} H(x_i,G)$ for any basis $\{x_i\}_{i=1}^n$ of $\mathbb{R}^n$. Choosing the standard basis yields $\sum_{i=1}^{n} x_i x_i^\top = I_n$, hence minimizing $\sum_{i=1}^n H(x_i,G)$ is equivalent to minimizing \eqref{math:obj_G}. 
    The constraint $I_n = \tilde{X} G$ ensures a valid CL parameterization, yielding the constrained problem \eqref{math:opt_pol_eval}.
\end{proof}

By Lemma~\ref{lem:G_K}, the solution of \eqref{math:opt_pol_eval} is not unique for $T>m+n$.
Before presenting the minimal norm solution, we first establish the following auxiliary result. Let $M \!\coloneq \! \tilde{U}^\top\! R \tilde{U}$ and $\Pi \! \coloneq \! I_T  - \! \tilde{X}^\dagger \! \tilde{X}$.

\begin{lemma}
    \label{lemma:aux}
    The Equations $(\Pi M \Pi)^\dagger \! = \!  (\tilde{U} \Pi)^\dagger R^{-1} ((\tilde{U} \Pi)^\dagger)^\top $ and $(\Pi M \Pi)^\dagger M = (\tilde{U}\Pi)^\dagger \tilde{U}$ hold.
\end{lemma}

The proof is provided in the Appendix~\ref{app:proof4}.

\begin{corollary}
    \label{col:1} 
    The  solution $\hat{G}$ of minimum Frobenius norm to the optimization problem \eqref{math:opt_pol_eval} is%
    \begin{align}
    \hat G & = \tilde{X}^\dagger - (\tilde U \Pi)^\dagger \bigl( \tilde{U} \tilde{X}^\dagger + R^{-1} (( \tilde{U} \Pi  )^\dagger)^\top (P \bar{X} )^\top \bigr), \label{math:pol_imp_sol} 
    \end{align} 
    where $\Pi = I_T - \tilde{X}^\dagger \tilde{X} $.
\end{corollary}
\begin{proof}
    Since \eqref{math:opt_pol_eval} is convex ($M\succeq 0$) with an affine constraint, the KKT conditions are necessary and sufficient for optimality. Let $\Lambda \in\mathbb{R}^{n\times n}$. The Lagrangian to \eqref{math:opt_pol_eval} is $
     \mathcal{L}  =  \operatorname{tr} \!\big(G^\top \! M G + 2 P \bar{X} G + Q  + \Lambda^{\!\top}\! (I_n - \tilde{X} G) \big).$  Stationarity requires 
    \begin{align}
     \frac{\partial \mathcal{L} }{\partial G } = 2 MG + 2 (P \bar{X})^\top - \tilde{X}^\top \Lambda = 0. \label{math:stat} 
    \end{align} 
    Left-multiplying \eqref{math:stat} by $\Pi$ yields 
    \begin{align}
      0 & 
      = \Pi M G + \Pi (P \bar{X}  )^{\!\top}\label{math:stat2} 
    \end{align} 
    as $\Pi \tilde{X}^\top = 0$. The constraint $\tilde{X} G = I_n $ has the general solution $\hat G = \tilde{X}^\dagger + \Pi V$ for arbitrary $V\in\mathbb{R}^{T \times n}$. Substituting $\hat G$ into \eqref{math:stat2} results in $0= \Pi M (\tilde{X}^\dagger + \Pi V ) + \Pi (P \bar{X} )^\top$, where its minimum-norm solution with respect to $V$ is $\hat V = - (\Pi M \Pi)^\dagger  \Pi  \bigl( M \tilde{X}^\dagger + (P\bar{X} )^\top \bigr)$. 
    Let $\Omega = \tilde{U} \Pi$.
By Lemma~\ref{lemma:aux}, $ (\Pi M \Pi)^\dagger \Pi = \Omega^\dagger R^{-1} (\Omega^\dagger)^\top \Pi = \Omega^\dagger R^{-1} (\Omega \Omega^\top)^{-1} \tilde{U} \Pi \Pi$ = $\Omega^\dagger R^{-1} (\Omega^\dagger)^\top$, since $\Pi$ is idempotent. 
    Substituting $\hat{V}$ into $\hat{G}$ and using the identity from Lemma~\ref{lemma:aux} implies $\hat{G}$ is \eqref{math:pol_imp_sol}.
\end{proof}

\begin{lemma}
    \label{lemma:same_imp} 
    Let $P\succeq 0$. The policies obtained from the model-based and CL-parameterized policy improvement step coincide. In particular, $K=R^{-1}B^\top P = -\tilde{U}\hat{G}$, where $\hat{G}$ is defined in \eqref{math:pol_imp_sol}.
\end{lemma}
\begin{proof}
Since $\hat{G}$ satisfies \eqref{math:stat}, substituting $\tilde{U}\hat{G} = -K$ and \eqref{math:ex_sys_int} into \eqref{math:stat} results in
\begin{align}
 0 & = - 2 \tilde{U}^\top R K + 2 \bigl(P (A\tilde X + B \tilde U ) \bigr)^\top - \tilde{X}^\top \Lambda  \\
 & = \tilde{X}^\top (2 A^\top P - \Lambda) + 2 \tilde{U}^\top(B^\top P - RK). \notag
\end{align}
By Assumption~\ref{ass:2}, $\tilde{X}$ and $\tilde{U}$ have full row rank and $\operatorname{im}(\tilde{X}^\top)\cap\operatorname{im}(\tilde{U}^\top)=\{0\}$, so both terms vanish, yielding $\Lambda=2A^\top P$ and $K=R^{-1}B^\top P$.
\end{proof}

With the policy improvement step established, the PI algorithm under the CL parameterization is given in Algorithm~\ref{alg:PI_closed}.
\begin{algorithm}[t]
\caption{PI under CL parameterization}
\label{alg:PI_closed}
\KwIn{$G_0\in\mathcal{G}$ }
\KwOut{Optimal $G^*\in\mathcal{G}^*$}
$k=0$ \\
\While{$G_k \ne G_{k+1}$}{
 Evaluate $G_k$ with \eqref{math:Pol_eval} to obtain $P_k$ \\
  Improve policy $u=\tilde{U}G_kx$ via \eqref{math:pol_imp_sol} to obtain $G_{k+1}$ \\
  $k=k+1$
}
\KwRet $G_k$
\end{algorithm}

\begin{theorem}
The sequence $\{G_k\}_{k=0}^\infty$ generated by Algorithm~\ref{alg:PI_closed} converges to $G^*\in\mathcal{G}^*$ such that $K^*=-\tilde{U} G^*$ is the optimal LQR gain. Moreover, $G_k\in\mathcal{G}$ and thus $K_k =-\tilde{U} G_k\in\mathcal{K}$ for all $k\in \mathbb{N}_0$.
\end{theorem}
\begin{proof}
    Algorithm~\ref{alg:PI_closed} is equivalent to Kleinman's Algorithm~\eqref{math:kleinman}, with the difference lying solely in the parameterization, expressing the policy via $G$ instead of $K$. 
    The equivalence of the policy evaluation and policy improvement  follows from Theorem~\ref{theo:closed} and Lemma~\ref{lemma:same_imp}, respectively.
    Consequently, both algorithms generate identical sequences $\{P_k\}_{k=0}^\infty$ and $\{K_k\}_{k=0}^\infty = \{-\tilde{U}G_k\}_{k=0}^\infty$. 
Convergence of $\{K_k\}_{k=0}^\infty$ to $K^*$ and the Hurwitz property of $A_{K_k}=\bar{X}G_k$ follow from Theorem~\ref{theo:PI}. Finally, the projection in \eqref{math:pol_imp_sol} enforces $I_n=\tilde{X}G_k$, which implies $G_k\in\mathcal{G}$ and thus $K_k\in\mathcal{K}$ for all $k$.\end{proof}

\begin{remark}
    \label{remark:eff} 
    Algorithm~\ref{alg:PI_closed} requires only solving a Lyapunov equation per iteration, since $\tilde{X}^\dagger, \Pi, M,$ and $(\tilde{U}\Pi)^\dagger$ are computed once. Unlike gradient descent methods, it does not require step size considerations.
\end{remark}

\begin{remark}
    In contrast to the data-driven PI from \cite{Lopez2023}, Algorithm~\ref{alg:PI_closed} does not require the solution of a generalized Sylvester-transpose matrix equation and uses all samples instead of only $n+m$ linearly independent columns of $\left[\begin{smallmatrix}\tilde{U} \\ \tilde{X}\end{smallmatrix}\right]$.
\end{remark}

\subsection{IRL Parameterization}
For completeness, the PI algorithm under IRL parameterization from \cite{JIANG20122699} is given in Algorithm~\ref{alg:PI_IRL}.

\begin{algorithm}[t]
\caption{PI under IRL parameterization \cite{JIANG20122699}}
\label{alg:PI_IRL}
\KwIn{$K_0\in\mathcal{K}$ }
\KwOut{Optimal $K^*$}
$k=0$ \\
\While{$P_k \ne P_{k+1}$}{
 Solve \eqref{math:LSE} using $K_k$ to obtain $P_k$ and $B^\top P_k$ \\
 Improve policy via $K_{k+1}=R^{-1}B^\top P_k$ \\
  $k=k+1$
}
\KwRet $P_k$
\end{algorithm}

\begin{theorem}[\texorpdfstring{\cite[Theorem 7]{JIANG20122699}}{}]
   The sequences $\{K_k\}_{k=0}^\infty$ and $\{P_k\}_{k=0}^\infty$ generated by Algorithm~\ref{alg:PI_IRL} converge to the optimal feedback $K^*$ and the value matrix $P^*$ for $k\to\infty$. Moreover, $K_k\in\mathcal{K}$ for all $k\in\mathbb{N}_0$.
\end{theorem}

We now compare Algorithms~\ref{alg:PI_closed} and \ref{alg:PI_IRL}. 
Algorithm~\ref{alg:PI_closed} evaluates the policy by solving an $n\times n$ Lyapunov equation, yielding a per-iteration cost of $\mathcal{O}(n^3)$ using standard dense linear algebra, while the policy improvement is computationally light after preprocessing (see Remark~\ref{remark:eff}). In contrast, Algorithm~\ref{alg:PI_IRL} performs policy evaluation via a least-squares problem with $d=\tfrac{n(n+1 )}{2}+mn$ unknowns, resulting in a complexity of $\mathcal{O} (T n^4)$, and requires higher sample complexity. Consequently, Algorithm~\ref{alg:PI_closed} is more computationally efficient than Algorithm~\ref{alg:PI_IRL} for larger state dimensions, since its pseudoinverses are computed once, whereas Algorithm~\ref{alg:PI_IRL} solves a new least-squares problem in each iteration.

\section{Riccati Flows and Value Iteration Algorithms}
\label{sec:CARE} 
This section presents data-driven formulations for the CARE, forming the basis for Riccati flows and VI algorithms.
The CL parameterization yields novel results, whereas the IRL-based formulations are from \cite{BIAN2016348}.

\subsection{Closed-loop Parameterization}
A natural approach to obtain a CARE is to substitute the improved policy \eqref{math:pol_imp_sol} into the policy evaluation~\eqref{math:Pol_eval}, yielding
\begin{align}
  &  \bigl(  \bar{X} ( I - (\tilde U \Pi)^\dagger \tilde U) \tilde{X}^\dagger \bigr)^\top \! P + P \bigl(   \bar{X} ( I - (\tilde U \Pi)^\dagger \tilde U) \tilde{X}^\dagger \bigr) \notag \\
  & \quad  - P \bar{X}(\tilde{U}\Pi )^\dagger R^{-1} \bigl( \bar{X}(\tilde{U}\Pi )^\dagger \bigr)^\top  P      \label{math:care_CL_simplified} + Q = 0 
\end{align}
with coefficients expressed entirely in terms of data. 
The structure of \eqref{math:care_CL_simplified} is analogous to the CARE \eqref{math:CARE}.
In the following, we establish that the coefficients in \eqref{math:care_CL_simplified} corresponding to $A$ and $B$ coincide with the true system matrices.

\begin{lemma}
    The equations $B=\bar{X}(\tilde{U}\Pi )^\dagger$ and
    $A= \bar{X} (I_T - (\tilde{U} \Pi)^\dagger \tilde{U}) \tilde{X}^\dagger$ hold.
\end{lemma}
\begin{proof}
    Right-multiplying $\Pi$ to \eqref{math:ex_sys_int} results in 
   $  \bar{X} \Pi = A \tilde{X} \Pi + B \tilde{U} \Pi = B \tilde U \Pi$,     because $\tilde X \Pi = 0$.  The full row rank of $\tilde{U}\Pi$ (see Proof of Lemma~\ref{lemma:aux}) implies $B = \bar{X} \Pi (\tilde{U}\Pi )^\dagger  = \bar{X} (\tilde U \Pi )^\dagger$ as $\Pi$ is idempotent.  Next, let  $E = I_T - (\tilde{U} \Pi)^\dagger \tilde{U}$. Then $\tilde U E  = \tilde U - \tilde{U} \Pi \tilde{U}^\top (\tilde{U}\Pi \tilde{U}^\top)^{-1} \tilde{U}= 0$. Finally,%
    \begin{align}
     \bar{X} E \tilde{X}^\dagger &= ( A \tilde{X} + B \tilde{U}) E \tilde{X}^\dagger = A \tilde{X} E \tilde{X}^\dagger  \\
     & = A \tilde{X} (I_T - \Pi \tilde{U}^\top (\tilde{U}\Pi (\tilde{U}\Pi)^\top)^{-1} \tilde{U})\tilde{X}^\dagger  = A \notag 
    \end{align}%
    because $\tilde{X}\Pi =0$ and $\tilde{X}\tilde{X}^\dagger = I_n$.
\end{proof}

\begin{remark}
    \label{remark:min_G}Note that \eqref{math:care_CL_simplified} is not the only valid CARE, any solution $G$ of \eqref{math:opt_pol_eval} substituted into \eqref{math:Pol_eval} yields one.
\end{remark}

We next verify that the obtained data-driven expressions for $A$ and $B$ coincide with their least-squares estimates, given by
\begin{align}
 \begin{bmatrix}  B_{\text{LS}} & A_{\text{LS}} \end{bmatrix} & = \bar{X} \begin{bmatrix}\tilde U \\ \tilde{X} \end{bmatrix}^\dagger \label{math:LS_id}. 
\end{align}
\begin{proposition}
    \label{prop:LS} 
    Let $   A  = \bar{X} (I_T - (\tilde{U} \Pi)^\dagger \tilde{U}) \tilde{X}^\dagger $ and $ B  = \bar{X} (\tilde{U} \Pi   )^\dagger $. Then, $A$ and $B$ coincide with the least-squares estimates $A_{\mathrm{LS}}$ and $B_{\mathrm{LS}}$ of \eqref{math:LS_id}, respectively.
\end{proposition}
\begin{proof}
Formulating $A$ and $B$ as  $\begin{bmatrix} B & A \end{bmatrix} = \bar{X} \begin{bmatrix} (\tilde U \Pi)^\dagger & ( I_T - (\tilde{U} \Pi)^\dagger \tilde{ U} ) \tilde{X}^\dagger\end{bmatrix}$ and comparing it with \eqref{math:LS_id}, yields the equivalent statement%
    \begin{align}
        \label{math:equiv} 
     \begin{bmatrix} \tilde{U} \\ \tilde{X} \end{bmatrix} \begin{bmatrix} (\tilde U \Pi)^\dagger & ( I_T - (\tilde{U} \Pi)^\dagger \tilde{ U} ) \tilde{X}^\dagger\end{bmatrix} =  I_{m+n}.
    \end{align} 
Substituting $(\tilde U \Pi)^\dagger = (\tilde U \Pi)^\top (\tilde U \Pi (\tilde U \Pi)^\top)^{-1}$, $\tilde{U}(\tilde{U}\Pi)^\dagger = I_m$, and $\tilde{X} (\tilde U \Pi)^\dagger =0$ 
into \eqref{math:equiv} confirms the identity.
\end{proof}%

\begin{remark}
    \label{remark:indirect} 
 CARE \eqref{math:care_CL_simplified} enables data-driven Riccati flow and VI, and convergence in the noise-free case follows from Theorems~\ref{theo:ric_flow} and~\ref{theo:VI}, respectively. However, since the data-driven expressions of $A$ and $B$ coincide with their least-squares estimates \eqref{math:LS_id}, this approach implicitly performs system identification and is therefore not direct. Consequently, the explicit Riccati flow and VI formulations are omitted.
\end{remark}

Inspired by \cite{LOPEZ2026113002}, the following  CARE can be constructed.
\begin{proposition}
    \label{prop:CARE_closed_2} 
    Let $P\in\mathbb{S}^n_+,~  F(P)= \tilde{X}^\top P \bar{X} + \bar{X}^\top P \tilde{X} + \tilde{U}^\top R \tilde{U} + \tilde{X}^\top Q \tilde{X} $ and define $J \in\mathbb{S}^{n+m}$ as 
    \begin{align}
     J = \begin{bmatrix} J_{11} & J_{12} \\ J_{21} & J_{22}\end{bmatrix} = \begin{bmatrix} \tilde{X}^\top & \tilde{U}^\top \end{bmatrix}^\dagger F(P) \begin{bmatrix} \tilde{X} \\ \tilde{U} \end{bmatrix} ^\dagger. \label{math:J} 
    \end{align} 
    A valid data-driven CARE is 
    \begin{align}
    0 & = J_{11}- J_{12} R^{-1} J_{21}. \label{math:CARE_closed} 
    \end{align} 
\end{proposition}
\begin{proof}
Pre- and post-multiplying the matrix $\tilde{J}$, whose Schur complement is the left-hand side of the CARE \eqref{math:CARE}, by $\left[\begin{smallmatrix}\tilde{X}^\top & \tilde{U}^\top\end{smallmatrix}\right] $ and its transpose yields
\begin{align}
{\begin{bmatrix} \tilde{X}^\top & \tilde U^\top \end{bmatrix}} \tilde{J}  \begin{bmatrix} \tilde{X}^\top & \tilde U^\top \end{bmatrix}^\top  \overset{ \eqref{math:ex_sys_int} }{=} F(P), \label{math:quad} 
\end{align}
where $\tilde{J} =  \left[\begin{smallmatrix} Q + PA + A^\top P & PB \\ B^\top P & R \end{smallmatrix}\right]$.
 By Assumption~\ref{ass:2}, $J$ is equal to $\tilde{J}$ and the CARE is recovered by \eqref{math:CARE_closed}.
\end{proof}

Equation \eqref{math:CARE_closed} enables a data-driven Riccati flow and VI, analogous to those in Theorems~\ref{theo:ric_flow} and \ref{theo:VI}, whose explicit formulations are omitted for brevity.

\begin{remark}
    \label{remark:recov} 
Once $P^*$ is obtained via Riccati flow or VI, the optimal feedback is recovered as $K^* = R^{-1} J_{21}^*$, where $J^*$ is the evaluation of \eqref{math:J} at $P^*$.
\end{remark}

\begin{remark}
    \label{remark:CARE_2} 
    Although \eqref{math:CARE_closed} relies on the data matrices of the CL parameterization, it does not utilize its decision variable $G$ and the equality $I_n = \tilde{X} G$. Consequently, the formulation of Proposition~\ref{prop:CARE_closed_2} does not strictly conform to the definition of the CL parameterization. Moreover, it can be shown that substitution of the least-squares estimate \eqref{math:LS_id} into $\tilde{J}$ results in \eqref{math:J}, thus system identification is still carried out implicitly. Nevertheless, simulation results of VI and Riccati flow based on \eqref{math:CARE_closed} are presented in the numerical results for comparison.
\end{remark}

\subsection{IRL Parameterization}
To ensure completeness, we state the Riccati flow and the VI under IRL parameterization from \cite{BIAN2016348}.
By using the solution to the linear equations \eqref{math:LSE2}, a CARE can be formulated as 
\begin{align}
    \label{math:CARE_IRL_1} 
 0 & = H - (K^+)^\top R K^+ + Q, 
\end{align} 
where $H$ and $K^+$ depend on $P$.
Thus, the Riccati flow is 
\begin{align}
    \label{math:Riccati_flow_IRL} 
 \dot P & = H - (K^+)^\top R K^+ + Q, \quad P(0)\in\mathbb{S}^n_+.
\end{align}

\begin{theorem}
    The trajectory $P(t)$ generated by \eqref{math:Riccati_flow_IRL} converges to $P^*$ for $t\to\infty$. Moreover, $K^+= K^*$ can be recovered by solving \eqref{math:LSE2} with $P^*$.
\end{theorem}
\begin{proof}
    By Assumption~\ref{ass:rank_3}, $H = A^\top P + P A $ and $K^+=R^{-1}B^\top P$ are uniquely determined by \eqref{math:LSE2} for any $P\succeq 0$. Thus, \eqref{math:CARE_IRL_1} is equivalent to the model-based CARE and Theorem~\ref{theo:ric_flow} ensures the convergence of $P(t)$ to $P^*$ for $t\to\infty$. Consequently, $K^+ = K^*$ holds at convergence.
\end{proof}

\begin{remark}
    In contrast to the policy gradient flow  \eqref{math:grad}, the Riccati flow \eqref{math:Riccati_flow_IRL} may generate unstable feedbacks $K^+(t) $ for some $t\geq 0$. Hence, \eqref{math:Riccati_flow_IRL} is not, in general, suitable for direct online implementations.  
    An advantage, however, is that it requires no stabilizing initial policy.
\end{remark}

The VI under IRL parameterization is given in Algorithm~\ref{alg:VI_IRL}.\begin{algorithm}[t]
\caption{Value Iteration under IRL \cite[Algorithm 2]{BIAN2016348}}
\label{alg:VI_IRL}
\KwIn{$P_0\succeq 0$ }
\KwOut{Optimal value matrix $P^*$, optimal gain $K^*$}
$k=q=0$\\
\While{$P_k \ne P_{k+1}$}{
 Solve \eqref{math:LSE2} w.r.t. $(H_k, K^+_k)$
 $\tilde{P}_{k+1} = P_k + \varepsilon_k \! \left(H_k - (K^+)^\top R K^+ + Q   \right)$ \\
 \eIf{$\tilde{P}_{k+1}\notin B_q$}{
    $P_{k+1} = P_0$, $q = q+1$
 }{    $P_{k+1} = \tilde{P}_{k+1}$
 }
  $k=k+1$
}
\KwRet $P_k,K^+_k$
\end{algorithm}
\begin{theorem}[\texorpdfstring{\cite[Theorem 4.2]{BIAN2016348}}{}]
    Let $\{\varepsilon_k\}_{k=0}^\infty$ and $\{B_q\}_{k=0}^\infty$ satisfy \eqref{math:vareps} and \eqref{math:B_q}, respectively.
    The sequences $\{P_k\}_{k=0}^\infty$ and $\{K^+_k\}_{k=0}^\infty$ generated by Algorithm~\ref{alg:VI_IRL} converge to $P^*$ and $K^*$, respectively.
\end{theorem}

To ensure convergence and boundedness of the iterates $P_k$ in Algorithm~\ref{alg:VI_IRL}, the update is reset to the initial iterate $P_0$ whenever it leaves the predefined compact set $B_q$ and the iteration continues with a reduced step size. The sequences $\{\varepsilon_k\}_{k=0}^\infty$ and $\{B_q \}_{k=0}^\infty$ are hyperparameters that influence the convergence rate and can be omitted if  $\{\varepsilon_k\}$ is chosen to guarantee $P_k \to P^*$. 

Comparing the VI (or Riccati flow) induced by the CL and IRL parameterizations with the corresponding CAREs in \eqref{math:CARE_closed} and \eqref{math:CARE_IRL_1} shows that both parameterizations require only a one-time computation of the pseudoinverses
 $\bigl(\begin{bmatrix} \tilde{X}^\top & \tilde{U}^\top \end{bmatrix}^\top \bigr)^\dagger$ and $\tilde{\Phi}^\dagger$, respectively, since $\tilde{\Phi}$ in \eqref{math:LSE2} does not depend on $P$, so that their per-iteration costs are comparable.

\section{Convex Optimization Problems}
\label{sec:conv_prob} 
This section presents convex optimization problems under CL and IRL parameterizations.
\subsection{Closed-loop Parameterization}
Similar to the data-driven discrete-time formulation in \cite{de2019formulas}, we present the optimization problem \eqref{math:LQR_feron} under the CL parameterization and transform it into a convex problem.

\begin{theorem}
    The optimal LQR feedback of the system \eqref{math:lin_sys} is $K^* = - \tilde{U} Z^* (Y^*)^{-1}$, where $Z^*\in\mathbb{R}^{T\times n}$ and $Y^*\in\mathbb{S}_+^n$ minimize the convex program%
    \begin{subequations}
        \label{math:opt_G2} 
    \begin{align}
 \min_{Y,Z,S}~ & \operatorname{tr}\left(Q Y  \right) + \operatorname{tr}\left(S  \right) \\
 \text{s.t.}~ &   \begin{bmatrix} S & R^{1/2} \tilde{U} Z \\
Z^\top \tilde{U}^\top R^{1/2} & Y \end{bmatrix} \succeq 0 \\
 & \bar{X} Z + Z^\top \bar{X}^\top + I \preceq 0 \\
 & Y = \tilde{X} Z,~ Y \succeq 0, 
\end{align} 
    \end{subequations}
with $S\in\mathbb{S}^m$. 
\end{theorem}
\begin{proof}
    Substituting \eqref{math:sys_closed} into \eqref{math:LQR_feron} yields 
\begin{subequations}
    \begin{align}
 \min_{Y,G} ~ & \operatorname{tr}\left(Q Y  \right) + \operatorname{tr} (R^{1/2} \tilde{U} G Y G^\top \tilde{U}^\top R^{1/2})\\
 \text{s.t.}~ &  \bar{X} G Y + Y (\bar{X} G )^\top  + I \preceq 0 \\
 & I_n = \tilde{X} G, ~Y \succeq 0.
\end{align} 
\end{subequations}
By substituting $G Y = Z $ and applying a Schur complement, we obtain the convex program \eqref{math:opt_G2}, where $G^* = Z^* (Y^*)^{-1}$ and consequently $K^* = - \tilde{U} Z^* (Y^*)^{-1}$.
\end{proof}

Although it is possible to formulate problem \eqref{math:LQR_boyd} based on the data-driven CARE \eqref{math:care_CL_simplified}, such a formulation is omitted due to its indirect nature
(see Remark~\ref{remark:indirect}). Instead, we formulate problem \eqref{math:LQR_boyd} without the minimal-norm constraint on $G$.

\begin{theorem}
    If $Q \succ 0$ or $(A, \sqrt{Q})$ is observable, the optimal LQR feedback and the Riccati matrix of system \eqref{math:lin_sys} are $K^* = - \tilde{U} Z^* {S^*}^{-1}$ and $P^* = {S^*}^{-1}$, where $Z^*\in\mathbb{R}^{T\times n}$ and $S^*\in\mathbb{S}_+^n$ minimize the convex program%
    \begin{subequations}
        \label{math:conv_closed_2} 
         \begin{align}
     \min_{Z,S}~& -\operatorname{tr}\left(S \right) \\
     \text{s.t.}~ &  \begin{bmatrix}Z^\top \bar{X}^\top + \bar{X} Z & Z^\top  \tilde{U}^\top & S^\top Q^{1/2} \\
     \tilde{U} Z & - R^{-1} & 0_{m,n} \\
    Q^{1/2} S & 0_{n,m} & - I_n \end{bmatrix} \preceq 0 \label{math:c00}  \\
    &  N^\top ( \tilde{U}^\top R  \tilde{U} ) Z + N^\top \bar{X}^\top = 0 \label{math:c10}\\
    & S = \tilde{X} Z,~ S\succeq 0 \label{math:c30} 
    \end{align} 
    \end{subequations}
    with $N \in \mathbb{R}^{T \times (T-n)}$ satisfying $\operatorname{im}(N) = \ker(\tilde{X})$.
\end{theorem}

\begin{proof}
The data-driven CARE without the minimal-norm constraint on $G$ (see Remark~\ref{remark:min_G}) is characterized by \eqref{math:Pol_eval}, \eqref{math:stat2}, and $I_n = \tilde{X}G$. 
These conditions are necessary and sufficient for optimality of the LQR problem, as they collectively enforce the HJB equation for $P\succeq 0$. 
Thus, the problem of computing $P^*$ can be formulated as the nonconvex program
\begin{align}
 \max_{P,G} \operatorname{tr}\left(P \right) \text{ s.t. } \eqref{math:Pol_eval},~ \eqref{math:stat2},~ I_n = \tilde{X}G \text{ and } P\succeq 0, \label{math:nonconvex} 
\end{align} 
which is subsequently relaxed to an SDP.
The condition \eqref{math:stat2} with $\Pi = I_T - \tilde{X}^\dagger\tilde X$ simplifies to $N^\top (\tilde{U}^\top R \tilde{U} ) G + N^\top \bar{X}^\top P =  0$. Since $P^*\succ 0$ by Lemma~\ref{lemma:pos_s}, pre- and post-multiplying  \eqref{math:Pol_eval} by $P^{-1}$ and substituting $Z= GP^{-1}$ and $P^{-1}=S$ yield
 \begin{align}
  Z^\top \bar{X}^\top +  \bar{X} Z + S^\top Q S + Z^\top \tilde{U}^\top R \tilde U Z = 0. \label{math:lyap_subst} 
 \end{align} 
 Applying these substitutions to \eqref{math:stat2} and $I_n = \tilde{X}G$ results in \eqref{math:c10} and \eqref{math:c30}, respectively. Since $Q\succeq 0$ and $\tilde{U}^\top  R \tilde{U}\succeq 0$, \eqref{math:lyap_subst} can be relaxed to an LMI without altering the solution of \eqref{math:nonconvex} and repeatedly applying the Schur complement then yields \eqref{math:c00}. 
 Finally, the resulting SDP is \eqref{math:conv_closed_2}.
\end{proof}

The additional condition ensures $P^*\succ 0$, which the substitution $S=P^{-1}$ requires, since under Assumption~\ref{ass:1} alone $P^*$ may be singular, in which case \eqref{math:conv_closed_2} is unbounded.
In contrast to Theorem~\ref{math:theo:dual}, \eqref{math:opt_G2} and \eqref{math:conv_closed_2} are not dual to each other, since \eqref{math:conv_closed_2} maximizes $\operatorname{tr}(P^{-1})$ instead of $\operatorname{tr}(P)$.
For completeness, we state the convex program from \cite{LOPEZ2026113002},  adapted to the integrated dynamics \eqref{math:sys_int} using the CARE of Proposition~\ref{prop:CARE_closed_2}.

\begin{theorem}[\texorpdfstring{\cite[Th. 16]{LOPEZ2026113002}}{}]
    The optimal Riccati matrix $P^*$ of the system \eqref{math:lin_sys}  is the solution of%
     \begin{subequations}
        \label{math:closed_conv_3} 
    \begin{align}
     \max_P ~& \operatorname{tr}\left( P \right) \\
     \text{s.t.}~ & F(P)\succeq 0,~ P\succeq 0,
    \end{align} 
\end{subequations}
    where $F(P)$ is defined as in Proposition~\ref{prop:CARE_closed_2}.
\end{theorem}

The optimal feedback $K^*$ can be computed from $P^*$ as outlined in Remark~\ref{remark:recov}.

\subsection{IRL Parameterization}
We derive a set of equations under the IRL parameterization that admits a CARE interpretation and differs from \eqref{math:CARE_IRL_1}. This enables a convex optimization formulation analogous to  \eqref{math:LQR_boyd}.
For brevity, define for each interval $[t_i,t_i+\delta]$
\begin{align}
 \Gamma^{\Delta x}_i \coloneqq  r_{\Delta x}(t_i), ~
 \Gamma^{xx}_i\coloneqq r_{xx}(t_i) \text{ and }\Gamma^{xu}_i\coloneqq r_{xu}(t_i). 
\end{align} 
Substituting $K=R^{-1}B^\top P$ into \eqref{math:int_lyap} yields, for $i=1,\dots,T$, 
\begin{align}
 \operatorname{tr}\left(P \Gamma_i^{\Delta x} \right)
 & = - \operatorname{tr}(Q \Gamma_i^{xx}) + \operatorname{tr}\left(P B R^{-1} B^\top P \Gamma_i^{xx} \right) \notag \\
 & \quad ~ \! + 2 \operatorname{tr}(B^\top P  \Gamma_i^{xu}). \label{math:data_CARE} 
\end{align}
which can be viewed as a data-driven CARE, where $P$ and $PB$ are unknown, as the improved policy is substituted into an equation that mirrors a policy evaluation (see Remark~\ref{remark_1}).
Next, let $W=PB$ and define  $f_i:\mathbb{S}^{n}\times \mathbb{R}^{n\times m} \to \mathbb{R}$ by 
\begin{align}
     (P,W) \mapsto  & f_i = 
 \operatorname{tr}\left(P \Gamma_i^{\Delta x} \right) + \operatorname{tr}(Q \Gamma_i^{xx}) \notag \\ 
 & \quad \! ~  - \operatorname{tr}(W  R^{-1} W^\top \Gamma_i^{xx}) 
  \!-\! 2 \operatorname{tr}\left( W^\top \Gamma_i^{xu} \right). \label{math:data_CARE_f} 
\end{align} 
Then $f_i(P,W)=0$ is equivalent to \eqref{math:data_CARE}.

\begin{lemma}
    \label{lemma:conv_f} 
    The function $f_i$ in \eqref{math:data_CARE_f} is concave in $(P,W)$.
\end{lemma}
\begin{proof}
    The terms $\operatorname{tr}(P\Gamma_i^{\Delta x})$ and $-2\operatorname{tr}(W^\top\Gamma_i^{xu})$ are affine in $P$
and $W$, and $\operatorname{tr}(Q\Gamma_i^{xx})$ is constant. 
Moreover,
\begin{align}
\operatorname{tr}\!\left(W R^{-1} W^\top \Gamma_i^{xx}\right)
= \big\| (\Gamma_i^{xx})^{1/2} W R^{-1/2}\big\|_F^2 \geq 0,
\end{align}
hence $-\operatorname{tr}(W R^{-1} W^\top \Gamma_i^{xx})$ is the negative of a convex quadratic, since $\Gamma_i^{xx}\succeq 0$ and $R\succ 0$. Therefore $f_i$ is concave as a sum of concave terms.
\end{proof}

Lemma~\ref{lemma:conv_f} implies that the equalities $f_i(P,W)=0$ are generally nonconvex, while the superlevel sets 
$\{(P,W)\mid f_i(P,W)\ge 0\}$ are convex. This motivates relaxing $f_i(P,W)=0$ via an epigraph variable to obtain a convex program analogous to \eqref{math:LQR_boyd}.

\begin{theorem}
    The LQR gain of the system \eqref{math:lin_sys} is $K^* = R^{-1} {W^*}^\top$, where $(P^*,W^*,Z^*)$ with $W^*\in \mathbb{R}^{n\times m}$ solves the convex program%
\begin{subequations}
        \label{math:LQR_data_2} 
         \begin{align}
     \max_{P,W,Z} ~& \operatorname{tr}\left(P \right)\\
     \text{s.t.} ~& \operatorname{tr}\left(P \Gamma_i^{\Delta x} \right) + \operatorname{tr}\left(Q \Gamma_i^{xx} \right) - \operatorname{tr}\left(Z \Gamma_i^{xx} \right) \notag \\ 
     &  - 2 \operatorname{tr}\left(W^\top \Gamma_i^{xu} \right)= 0 \quad i=1,\dots,T  \label{math:scalar_c} \\
     & \begin{bmatrix} Z & W \\ W^\top & R \end{bmatrix}\succeq 0,~P\succeq 0 \label{math:shur_c} 
    \end{align} 
    \end{subequations}
    with $Z\in\mathbb{S}^n$.
\end{theorem}

\begin{proof}
By \eqref{math:alt_1}, the left-hand side of \eqref{math:scalar_c} equals $\operatorname{tr}\bigl((A^\top P+PA+Q-Z)\Gamma_i^{xx}\bigr)+2\operatorname{tr}\bigl((B^\top P-W^\top)\Gamma_i^{xu}\bigr)$. Stacking these expressions for $i=1,\dots,T$ shows that \eqref{math:scalar_c} is equivalent to
\begin{align}
 \begin{bmatrix} \Gamma^{xx}D & 2\Gamma^{ux} \end{bmatrix} \begin{bmatrix} \operatorname{vech}(A^\top P+PA+Q-Z) \\ \operatorname{vec}(B^\top P-W^\top) \end{bmatrix} = 0,
\end{align}
where the matrix has full column rank by Assumption~\ref{ass:rank_2}. Hence, $Z=A^\top P+PA+Q$ and $W=PB$. Substituting both into \eqref{math:shur_c} yields \eqref{math:LQR_boyd_care}, so that \eqref{math:LQR_data_2} is equivalent to \eqref{math:LQR_boyd} and attains $P^*$ with $W^*=P^*B$, which yields $K^*=R^{-1}(W^*)^\top$.
\end{proof}

The parameterization \eqref{math:LSE2} also yields a convex program.

\begin{theorem}
    The optimal LQR gain of the system \eqref{math:lin_sys} is $K^* = {K^+}^*$, where ${K^+}^*$ is the solution of the convex program%
    \begin{subequations}
        \label{math:IRL_conv_2} 
            \begin{align}
 \max_{P,H,K^+} & \operatorname{tr}\left( P  \right)\\
 \text{s.t. } &  \begin{bmatrix} H + Q & (K^+) ^\top \\ K^+ & R^{-1}\end{bmatrix} \succeq 0 \label{math:CARE_Schur} \\
 & \eqref{math:LSE2},~ P\succeq 0,
\end{align} 
    \end{subequations}
    with $H\in\mathbb{S}^n$.
\end{theorem}
\begin{proof}
    For any $P\succeq 0$, the equality \eqref{math:LSE2} uniquely determines $K^+ = R^{-1}B^\top P$ and $H = A^\top P + PA$ under Assumption~\ref{ass:rank_3}. Since $R\succ 0$, the Schur complements of \eqref{math:CARE_Schur} with respect to $R^{-1}$ and of \eqref{math:LQR_boyd_care} with respect to $R$ both equal $A^\top P+PA+Q-PBR^{-1}B^\top P$. Hence, \eqref{math:CARE_Schur} is equivalent to \eqref{math:LQR_boyd_care}, implying the equivalence of \eqref{math:IRL_conv_2} and \eqref{math:LQR_boyd}.
\end{proof}

Problems \eqref{math:LQR_data_2} and \eqref{math:IRL_conv_2} correspond, respectively, to the fundamental equations \eqref{math:LSE} and \eqref{math:LSE2} of the IRL parameterization (see Remark~\ref{remark:IRL}), and both are data-driven counterparts of the convex program \eqref{math:LQR_boyd}.
The relation \eqref{math:alt_1} also yields a counterpart of \eqref{math:LQR_feron}, which is the IRL analogue of \eqref{math:opt_G2}.

\begin{theorem}
    The optimal LQR feedback of the system \eqref{math:lin_sys} is $K^*=-U_{\nu^*}(Y^*)^{-1}$, where $\nu^*\in\mathbb{R}^T$ and $Y^*\in\mathbb{S}_+^n$ minimize the convex program%
    \begin{subequations}
        \label{math:IRL_conv_3}
    \begin{align}
 \min_{Y,\nu,S_\nu}~ & \operatorname{tr}(QY) + \operatorname{tr}(S_\nu) \\
 \text{s.t.}~ & \begin{bmatrix} S_\nu & R^{1/2}U_\nu \\ U_\nu^\top R^{1/2} & Y \end{bmatrix} \succeq 0 \label{math:IRL_LMI} \\
 & \textstyle\sum_{i=1}^T\nu_i\Gamma_i^{\Delta x} + I_n \preceq 0 \label{math:IRL_Lyap} \\
 & Y = \textstyle\sum_{i=1}^T\nu_i\Gamma_i^{xx},~ Y\succeq 0,
    \end{align}
    \end{subequations}
    with $U_\nu\coloneqq\sum_{i=1}^T\nu_i(\Gamma_i^{xu})^\top$ and $S_\nu\in\mathbb{S}^m$.
\end{theorem}
\begin{proof}
Multiplying \eqref{math:alt_1} at $t=t_i$ by $\nu_i\in\mathbb{R}$ and summing over $i$ yields
\begin{align}
 \textstyle\sum_{i=1}^T\nu_i\Gamma_i^{\Delta x} = AY+YA^\top+BU_\nu+U_\nu^\top B^\top \label{math:IRL_sum}
\end{align}
with $Y=\sum_{i=1}^T\nu_i\Gamma_i^{xx}$. The linear map $\nu\mapsto(Y,U_\nu)$ is represented by $\left[\begin{smallmatrix}\Gamma^{xx} & \Gamma^{ux}\end{smallmatrix}\right]^\top$, whose rank equals the dimension $\tfrac{n(n+1)}{2}+mn$ of $\mathbb{S}^n\times\mathbb{R}^{m\times n}$ by Assumption~\ref{ass:rank_2}. Hence, for every $Y\in\mathbb{S}^n$ and $K\in\mathbb{R}^{m\times n}$, there exists $\nu$ with $Y=\sum_{i=1}^T\nu_i\Gamma_i^{xx}$ and $U_\nu=-KY$, for which \eqref{math:IRL_sum} becomes $A_KY+YA_K^\top=\sum_{i=1}^T\nu_i\Gamma_i^{\Delta x}$. Substituting this parameterization into \eqref{math:LQR_feron} and introducing $S_\nu\succeq R^{1/2}U_\nu Y^{-1}U_\nu^\top R^{1/2}$, which is equivalent to \eqref{math:IRL_LMI} by the Schur complement, yields \eqref{math:IRL_conv_3}. Every feasible point satisfies $Y\succ 0$, since $Yv=0$ implies $U_\nu v=0$ by \eqref{math:IRL_LMI} and thus $v^\top\bigl(\sum_{i=1}^T\nu_i\Gamma_i^{\Delta x}\bigr)v=0$ by \eqref{math:IRL_sum}, which contradicts \eqref{math:IRL_Lyap} for $v\neq 0$. Consequently, $K^*=-U_{\nu^*}(Y^*)^{-1}$.
\end{proof}

\begin{table}
       \caption{Comparison of Convex Problems under CL and IRL parameterizations}\label{tab:comp}
\centering
   \setlength{\tabcolsep}{3pt}%
\begin{tabular}{cccccc}
\toprule
& Problem & \makecell{\# scalar \\ variables} & \makecell{\# equality \\ constraints} & LMI size \\
\midrule
\multirow{3}{*}{CL} & \eqref{math:opt_G2} & $\tfrac{n(n+1)}{2}+\tfrac{m(m+1)}{2}+Tn$ & $n^2$ & $n,\ n+m$ \\[2pt]
& \eqref{math:conv_closed_2} & $\tfrac{n(n+1)}{2}+Tn$ & $Tn$ & $n,\ 2n+m$ \\[2pt]
& \eqref{math:closed_conv_3} & $\tfrac{n(n+1)}{2}$ & -- & $n,\ T$ \\
\midrule
\multirow{3}{*}{IRL} & \eqref{math:LQR_data_2} & $n(n+1)+mn$ & $T$ & $n,\ n+m$ \\[2pt]
& \eqref{math:IRL_conv_2} & $n(n+1)+mn$ & $T$ & $n,\ n+m$ \\[2pt]
& \eqref{math:IRL_conv_3} & $\tfrac{n(n+1)}{2}+\tfrac{m(m+1)}{2}+T$ & $\tfrac{n(n+1)}{2}$ & $n,\ n+m$ \\
\bottomrule
\end{tabular}
\end{table}
Table~\ref{tab:comp} compares the problems \eqref{math:opt_G2}, \eqref{math:conv_closed_2}, \eqref{math:closed_conv_3}, \eqref{math:LQR_data_2}, \eqref{math:IRL_conv_2}, and \eqref{math:IRL_conv_3} in terms of the number of  decision variables, equality constraints, and LMI dimensions.
Problem \eqref{math:opt_G2} involves relatively small LMIs but a larger number of scalar variables, while problem \eqref{math:conv_closed_2} reduces the number of variables at the expense of a larger LMI and additional equality constraints, which may increase computational burden for large datasets. Problem \eqref{math:closed_conv_3} has the fewest decision variables, but its LMI dimension scales with the sample size, limiting its scalability. In contrast, the IRL-based formulations \eqref{math:LQR_data_2} and \eqref{math:IRL_conv_2} maintain LMIs of fixed size independent of 
$T$, at the cost of equality constraints that grow linearly with the data. Problem \eqref{math:IRL_conv_3} also has LMIs of fixed size and a fixed number of equality constraints, but, as in problem \eqref{math:opt_G2}, its number of scalar variables grows with $T$, albeit by $T$ instead of $Tn$. Since the number of interior-point iterations is nearly independent of the problem size \cite[Sec.~5.7]{Vandenberghe1996SemidefiniteP}, whereas the cost per iteration is dominated by the solution of a linear system whose size is determined by the number of scalar variables and the LMI dimensions \cite[Sec.~7.6]{Vandenberghe1996SemidefiniteP}, the IRL parameterization is expected to scale more favorably for large datasets. However, this advantage comes at the expense of increased data requirements, as the IRL parameterization generally requires more samples than the CL parameterization (see Section~\ref{subsec:comp}). 

\section{Policy Gradient Flows}
\label{sec:grad_flow} 
This section presents policy gradient flows under both parameterizations. While the formulation under CL parameterization is adapted to the continuous-time case from \cite{Zhao2023}, the IRL-based formulation yields novel results. 
\subsection{Closed-loop Parameterization}
In this subsection, we derive a closed-form expression for the gradient of the LQR cost under CL parameterization. Moreover, we show that the gradient projected onto the tangent space of $\mathcal{G}$ induces a projected gradient flow that achieves optimal LQR performance. 

\begin{proposition}
    Let $G\in\mathcal{G}$. Then, the gradient of the LQR cost \eqref{math:f_G} with respect to $G$  is
   \begin{align}
      \nabla f_G & = 2 \bigl(\tilde{U}^\top R \tilde U G + \bar{X}^\top P_G  \bigr) Y_G  \label{math:grad_G}
   \end{align}
   where $P_G$ and $Y_G\in\mathbb{S}^{n}_+$ satisfy \eqref{math:Pol_eval} and
   \begin{align}
      0 & = \bar{X} G Y_G + Y_G (\bar{X} G)^\top + I_n, \label{math:Y_G}
   \end{align}
   respectively.
\end{proposition}

\begin{proof}
   The matrix function \eqref{math:f_G} is a composition of $C^\omega$ maps, $G \mapsto P_G \mapsto \operatorname{tr }(P_G)$, see Lemma~\ref{lemma:f_G_ana}.
   The differential of \eqref{math:Pol_eval} with respect to $G$ is given by 
   \begin{align}
     0 & = (\bar{X} G )^\top dP_G + dP_G \bar{X} G + \Theta^\top + \Theta \label{math:ly4}, 
   \end{align}
   where $\Theta = dG^\top \bigl(\tilde{U}^\top R \tilde U G + \bar{X}^\top P_G  \bigr)$.
Since $\bar{X} G$ is Hurwitz  for every $G\in\mathcal{G}$, the solution of \eqref{math:ly4} is
$  dP_G = \int_{	0}^{\infty}  e^{ (\bar{X}G )^\top t}  ( \Theta +\Theta^\top  ) e^{ \bar{X} G  t} \mathrm{d}t.$
Thus, the differential of the LQR cost is given by 
\begin{align}
    d f_G  &= \operatorname{tr }(dP_G)  = \operatorname{tr}\left(\int_{	0}^{\infty}  e^{ (\bar{X}G )^\top t}  ( \Theta + \Theta^\top  ) e^{ \bar{X} G  t} \mathrm{d}t\right) \notag \\
    & = 2 \operatorname{tr}\bigl(\Theta \int_{	0}^{\infty}  e^{ \bar{X}G t}  e^{ (\bar{X} G)^\top  t} \mathrm{d}t  \bigr). \label{math:dFG}
\end{align}
The integral in \eqref{math:dFG} is the solution to \eqref{math:Y_G}, since $\bar{X} G $ is Hurwitz.
Thus, $df_G = 2 \operatorname{tr}(\Theta Y_G )$,
yielding \eqref{math:grad_G}.
\end{proof}

To maintain feasibility of the gradient flow, the projection $\Pi = I_T - \tilde{X}^\dagger \tilde{X}$ is used to premultiply the gradient $\nabla f_G$, thereby projecting it onto the tangent space of $\mathcal{G}$, i.e., $\operatorname{ker}(\tilde{X})$.
Moreover, let $\tilde{\Pi}\coloneq I_T - \left[\begin{smallmatrix}\tilde{U} \\ \tilde{X}\end{smallmatrix}\right]^\dagger \left[\begin{smallmatrix}\tilde{U} \\ \tilde{X}\end{smallmatrix}\right]$ denote the orthogonal projection onto $\ker\bigl(\left[\begin{smallmatrix}\tilde{U} \\ \tilde{X}\end{smallmatrix}\right]\bigr)$, so that $\tilde{\Pi}G$ is the component $N$ of $G$ in Lemma~\ref{lem:G_K}, which alters neither $K$ nor $A_K$, and let $G_K^p$ denote $G^p$ in \eqref{math:gen_sol} for a given $K$.

\begin{theorem}
    \label{theo:proj_grad_CL}
    Let $H_d\coloneq\tilde{U}\Pi\tilde{U}^\top$. Then $H_d\succ 0$, and the projected gradient flow
    \begin{align}
        \label{math:proj_grad_1}
     \dot G & = - \alpha_{\mathrm{CL}} \Pi \nabla f_G, \quad G(0)\in\mathcal{G}, \quad \alpha_{\mathrm{CL}}>0,
    \end{align}
    generates a trajectory $G(t)\in\mathcal{G}$ for all $t\geq 0$, which satisfies $\tilde{\Pi}G(t)=\tilde{\Pi}G(0)$ and whose gain $K(t)=-\tilde{U}G(t)$ solves the preconditioned policy gradient flow
    \begin{align}
        \dot{K} = -\alpha_{\mathrm{CL}} H_d\nabla f_K, \quad K(0)=-\tilde{U}G(0). \label{math:proj_grad_K}
    \end{align}
    Moreover, $K(t)$ and $G(t)$ converge to $K^*$ and $\tilde{\Pi}G(0)+G_{K^*}^p\in\mathcal{G}^*$, respectively.
\end{theorem}

\begin{proof}
    Let $G\in\mathcal{G}$ and $K=-\tilde{U}G$. By Lemma~\ref{lem:G_K}, $A_K=\bar{X}G$, so that \eqref{math:Pol_eval} and \eqref{math:Y_G} coincide with \eqref{math:Lyap1} and with \eqref{math:Lyap2} for $x_0x_0^\top$ replaced by $I_n$, respectively. Hence, $P_G=P_K$ and $Y_G=Y_K$, and substituting \eqref{math:ex_sys_int} into \eqref{math:grad_G} and using \eqref{math:grad_mod} yields
    \begin{align}
        \nabla f_G = -\tilde{U}^\top\nabla f_K + 2\tilde{X}^\top A^\top P_KY_K. \label{math:grad_G_K}
    \end{align}
    Since $\Pi\tilde{X}^\top=0$, \eqref{math:grad_G_K} gives $\Pi\nabla f_G=-\Pi\tilde{U}^\top\nabla f_K$. Thus, \eqref{math:proj_grad_1} reads $\dot{G}=\alpha_{\mathrm{CL}}\Pi\tilde{U}^\top\nabla f_K$, which implies $\tilde{X}\dot{G}=0$, the gain dynamics \eqref{math:proj_grad_K}, and $\tilde{\Pi}\dot{G}=0$, because $\tilde{\Pi}\Pi=\tilde{\Pi}$ by $\ker\bigl(\left[\begin{smallmatrix}\tilde{U} \\ \tilde{X}\end{smallmatrix}\right]\bigr)\subseteq\ker(\tilde{X})$ and $\tilde{\Pi}\tilde{U}^\top=0$. Moreover, $H_d=(\tilde{U}\Pi)(\tilde{U}\Pi)^\top\succ 0$, since $v^\top\tilde{U}\Pi=0$ implies $\tilde{U}^\top v\in\ker(\Pi)=\operatorname{im}(\tilde{X}^\top)$, which forces $v=0$ by Assumption~\ref{ass:2}.

    Let $V(K)=f_K-f_{K^*}$ be the Lyapunov candidate. Then $V(K^*)=0$ and $V(K)>0$ for all $K\neq K^*$ over $\mathcal{K}$. The time derivative of $V(K)$ along \eqref{math:proj_grad_K} is
    \begin{align}
        \dot{V}(K) &= \operatorname{tr}\bigl((\nabla f_K)^\top\dot{K}\bigr) = -\alpha_{\mathrm{CL}}\operatorname{tr}\bigl((\nabla f_K)^\top H_d\nabla f_K\bigr) \notag \\
        &= -\alpha_{\mathrm{CL}}\bigl\Vert H_d^{1/2}\nabla f_K\bigr\Vert_F^2 \leq 0.
    \end{align}
    Since $H_d\succ 0$ and $\nabla f_K$ in \eqref{math:grad_mod} with $Y_K\succ 0$ vanishes only at $K^*$, $\dot{V}(K)<0$ for all $K\in\mathcal{K}\setminus\{K^*\}$, so that $V(K)$ is a Lyapunov function, implying asymptotic stability of $K^*$. Moreover, $K(t)$ remains in the compact sublevel set $\{K\in\mathcal{K}\mid f_K\leq f_{K(0)}\}$, because $V(K)$ is nonincreasing along trajectories. Thus, $K(t)\to K^*$, and by Lemma~\ref{lem:G_K} and $\tilde{\Pi}G(t)=\tilde{\Pi}G(0)$, the trajectory $G(t)=G_{K(t)}^p+\tilde{\Pi}G(0)$ remains in $\mathcal{G}$ and converges to $\tilde{\Pi}G(0)+G_{K^*}^p$, which lies in $\mathcal{G}^*$ by Lemma~\ref{lem:G_star}.
\end{proof}

Since $Y_G\succ 0$, the steady-state condition $\Pi\nabla f_G=0$ of \eqref{math:proj_grad_1} coincides with the optimality condition \eqref{math:stat2} for $P=P_G$ associated with the HJB equation. By \eqref{math:grad_G_K}, it is equivalent to $H_d\nabla f_K=0$, which by $H_d\succ 0$ holds only for $K=K^*$. Hence, the equilibrium set of \eqref{math:proj_grad_1} is $\mathcal{G}^*$ by Lemma~\ref{lem:G_star}, which reduces to a single point for $T=m+n$ and contains infinitely many points for $T>m+n$.

\begin{remark}
    A regularizer can be added to the objective function, i.e.,
     \begin{align}
     f_G^{\lambda} = f_G +   \lambda\Vert \tilde{\Pi} G \Vert, \quad\lambda>0. 
    \end{align}
    This regularizer enforces the orthogonality constraint $\tilde{\Pi}G = 0$, thereby selecting the unique minimum-norm solution.
    Such regularization has been studied in the discrete-time case, where it promotes certainty equivalence and robustness to process noise \cite[Sec.~III.A]{dorfler2022}. 
      Moreover, for projected gradient descent in discrete-time LQR, implicit regularization occurs when $G(0)$ satisfies $\tilde{\Pi} G(0) = 0$ \cite[Theorem 2]{Zhao2023}. For noise-free data, Theorem~\ref{theo:proj_grad_CL} establishes the continuous-time counterpart, since \eqref{math:proj_grad_1} preserves $\tilde{\Pi}G$ without regularizer. In particular, $G(0)=G_{K_0}^p$ yields $\lim_{t\to\infty} G(t)=G_{K^*}^p\in\mathcal{G}^*$ with $\tilde{\Pi}G_{K^*}^p=0$.
\end{remark}

\subsection{IRL Parameterization}
In this subsection, we derive the policy gradient flow under IRL parameterization. Before stating the gradient expression, we first establish the following auxiliary lemma.

\begin{lemma}
    \label{math:lemma_L} 
    Let $\Psi = \begin{bmatrix}I_{n^2} & 0_{n^2, nm}\end{bmatrix}$ and $L\in\mathbb{R}^{n\times n}$ be defined by $\operatorname{vec}\left(L  \right) \coloneq \bigl( \Psi \bar{\Phi}^\dagger \Gamma^{xx} \bigr)^\top \operatorname{vec}\left(I_n \right)$, where $\bar{\Phi}$ is as in \eqref{math:LSE_0}. Then, the matrix $L$ is symmetric.
\end{lemma}
\begin{proof}
    It is sufficient to prove $\operatorname{vec}\left(L^\top  \right) = C_{n,n} \operatorname{vec}\left(L \right) = \operatorname{vec}\left( L  \right)$, where $C_{n,n}\in\mathbb{R}^{n^2 \times n^2}$ is a commutation matrix. Since $r_{xx}(t)=r_{xx}^\top(t)$ by definition, it follows $C_{n,n} \operatorname{vec}\left( r_{x x}(t)  \right) = \operatorname{vec}\left(r_{x x}(t) \right)$,
    implying $ \Gamma^{xx} C_{n,n}= \Gamma^{xx}$, and $C_{n,n} (\Gamma^{xx})^\top = (\Gamma^{xx})^\top$. Then, 
    \begin{align}
     \operatorname{vec}\left( L^\top  \right) &= C_{n,n} \operatorname{vec}\left( L  \right) = C_{n,n} (\Gamma^{xx})^\top (\Psi \bar{\Phi}^\dagger)^\top \operatorname{vec}\left(I_n  \right) \notag  \\
     & =  (\Gamma^{xx})^\top (\Psi \bar{\Phi}^\dagger)^\top \operatorname{vec}\left(I_n  \right) = \operatorname{vec}\left(L \right), 
    \end{align} 
    completing the proof.
\end{proof}

\begin{proposition}
    Let $K\in\mathcal{K}$. Then, the gradient of the LQR cost \eqref{math:f_IRL} with respect to $K$ is 
    \begin{align}
     \nabla_K \hat{f}_K = 2 \bigl( (B^\top P_K) - RK  \bigr) L , \label{math:grad_IRL} 
    \end{align} 
    where $B^\top P_K $ is the solution to \eqref{math:LSE} and $L$ is as in Lemma~\ref{math:lemma_L}.
\end{proposition}

\begin{proof}
    The differential of \eqref{math:LSE_0} with respect to $K$ is 
    \begin{align}
     &\bar{\Phi} \begin{bmatrix} \operatorname{vec}\left(d P  \right)\\ \operatorname{vec}\left( d (B^\top P ) \right)\end{bmatrix}  \\ & ~~ = \Gamma^{xx} \operatorname{vec}\left( - dK^\top R K - K^\top R dK + 2  dK^\top (B^\top P) \right). \notag
    \end{align}
    Consequently, the differential of $\hat{f}_K$ \eqref{math:f_IRL} is
    \begin{align}
     d\hat{f}_K &= \operatorname{tr}\left(dP  \right) = \operatorname{vec}\left(I_n  \right)^\top \operatorname{vec}\left( dP  \right) \\
     & = \operatorname{tr}\left(L^\top \left(- dK^\top R K - K^\top R dK +  2  dK^\top (B^\top P)  \right) \right)  \notag \\
     & = \operatorname{tr}\left( dK^\top (- RK L^\top - RK L + 2 (B^\top P) L^\top   ) \right), \notag
    \end{align} 
    implying \eqref{math:grad_IRL}, since $L$ is symmetric by Lemma~\ref{math:lemma_L}.
\end{proof}

\begin{lemma}
    The cost function $\hat{f}_K$ \eqref{math:f_IRL} admits  a unique stationary point over $\mathcal{K}$ that coincide with the optimal LQR gain $K^*$. Moreover, $L = - Y_K$, where $Y_K$ is defined in \eqref{math:grad_mod}.
\end{lemma}
\begin{proof}
    By definition, $\hat{f}_K \equiv f_K$ over $\mathcal{K}$ (see Remark~\ref{remark:same_P}), and hence both functions share the same stationary points.    
    Since $Y_K \succ 0$ for all $K\in\mathcal{K}$, the condition $\nabla f_K = 0 $ holds if and only if $RK - B^\top P_K = 0$ which is satisfied uniquely at $K^* = R^{-1} B^\top P^*$. 
    Comparing the gradient expressions \eqref{math:grad_mod} and \eqref{math:grad_IRL}, and using $\hat{f}_K \equiv f_K$, implies $L = - Y_K$.
\end{proof}

\begin{theorem}
    The gradient flow 
    \begin{align}
     \dot K = - \alpha_{\mathrm{IRL}} \nabla \hat{f}_K, \quad K(0)\in\mathcal{K}, \label{math:flow_IRL} 
    \end{align}
    where $\alpha_{\mathrm{IRL}} >0$, 
    solves the optimization problem~\eqref{math:f_hat} and converges to the optimal feedback $K^*$. The trajectory $K(t)$ remains within $\mathcal{K}$ for all $t\geq 0$. 
\end{theorem}
\begin{proof}
Let $V(K) = \hat f_K - \hat f_{K^*}$ be the Lyapunov candidate. Then,  $V(K^*)=0$ and $V(K)>0$ for all $K\neq K^*$ over $\mathcal{K}$. The time derivative of $V(K)$ along \eqref{math:flow_IRL} is 
\begin{align}
 \dot V(K) = \operatorname{tr}\bigl( (\nabla \hat{f}_K )^\top \dot{K} \bigr) = -\alpha_{\mathrm{IRL}} \Vert \nabla \hat{f}_K \Vert_F^2 \leq 0. \label{math:dV} 
\end{align} 
Since $\dot V(K) <0 $ for all $K\in\mathcal{K}\setminus \{K^*\}$, $V(K)$ is a Lyapunov function, implying asymptotic stability of $K^*$. 
Moreover,  $K(t)$ remains in the compact sublevel set $\{K\in\mathcal{K} \mid \hat f_K \leq \hat f_{K(0)} \}$ for any $K(0)\in\mathcal{K}$, because $V(K)$ is nonincreasing along trajectories. Thus, the region of attraction of $K^*$ is $\mathcal{K}$.
\end{proof}

\begin{remark}
    The convergence of the gradient flows \eqref{math:proj_grad_1} and \eqref{math:flow_IRL} implies convergence of the corresponding gradient descent iterations under an appropriate choice of step size.
\end{remark}

\section{Discussion and Numerical Results}
\label{sec:comp} 
This section provides a concluding discussion of the approaches and  presents simulation results.

\subsection{Discussion}
\begin{table}[t]
\caption{Comparison of Solution Approaches for LQR}
\label{tab:comparison_methods}
\centering
\setlength{\tabcolsep}{2pt}%
\begin{tabular}{ccccccc}
\toprule
 & Comp. & Conv. & Stable Initial & Online & One-shot \\
Method & Effort & Rate & Policy  & Capability & Solution \\
\midrule
PI      & Moderate & Fast     & Yes & Limited  & No \\
Riccati Flow               & Moderate & Moderate & No  & No  & No \\
VI    & Low      & Slow     & No  & No  & No \\
Convex Program             & High     & Fast     & No  & No  & Yes \\
Gradient Flow       & Moderate & Moderate & Yes & Yes & No \\
\bottomrule
\end{tabular}
\end{table}

Table~\ref{tab:comparison_methods} summarizes the considered solution approaches in terms of computational effort, convergence rate, and implementation properties. Although all methods recover the same optimal gain, they differ significantly in their numerical characteristics and applicability. The classification applies uniformly to both the CL and IRL parameterizations, since both now admit each algorithm type, and the differences between the parameterizations manifest primarily in sample complexity and numerical accuracy, as demonstrated in Section~\ref{sec:numerical}.
In particular, PI achieves fast convergence at moderate computational cost but requires a stabilizing initial policy, whereas VI offers low per-iteration effort and requires no stabilizing initial policy, since it starts from $P_0=0$, at the expense of slow convergence. Riccati flow provides a continuous-time alternative with moderate complexity and convergence behavior, which also starts from $P(0)=0$. Since the gains of VI and the Riccati flow are not guaranteed to be stabilizing before convergence, neither is suited to closed-loop operation. Convex programs yield a one-shot solution without any initial policy, but at the highest computational effort, which grows with the number of samples under the CL parameterization.
The notion of online capability refers to whether an approach can be applied in closed-loop while maintaining overall stability \cite{Zhao2025}. The gradient flow generates stabilizing feedbacks with smooth evolution along the entire trajectory, which makes it suited to online adaptation, whereas the gain of PI changes in discrete steps between the stabilizing iterates, which may induce undesirable transients and complicate the analysis and certification of stability in closed loop.
Overall, the approaches are complementary and enable different trade-offs between efficiency, convergence, and implementation.

In practice, direct data-driven approaches must account for uncertainties such as process and measurement noise. An important question is whether the controller obtained from noisy data continues to stabilize the system, and how much performance degradation in terms of the LQR cost is expected \cite{dorfler2023}. Stability and suboptimality guarantees could be improved by extending the proposed methods with suitable regularization techniques \cite{dorfler2022}.
A detailed robustness analysis is beyond the scope of this paper and is left for future work.

\subsection{Numerical Results}
\label{sec:numerical} 
For the simulation results, 100 systems with $n=4,m=2$ were generated by sampling $A$ and $B$ matrices from a standard Gaussian distribution with about 60\% of the entries set to zero. The weighting matrices are $Q = I_4$ and $R = I_2$.
For each system, an initial stabilizing feedback gain $K_0$ was sampled from a standard Gaussian distribution until all eigenvalues of $A-BK_0$ had real parts smaller than $-0.1$.
The data matrices $\bar{X}, \tilde{U}, \tilde{X}$ and $\Gamma^{\Delta x},\Gamma^{xx}, \Gamma^{ux}$ were computed from the same $T=30$ samples, each a separate trajectory of length $\delta = \SI{0.1}{\second}$ with an initial state and an input drawn from standard Gaussian distributions, where the input is held constant for $\SI{0.01}{\second}$.
As an ODE solver, we used \texttt{ode45} in MATLAB with relative and absolute tolerances of $10^{-12}$ for the data and of $10^{-10}$ and $10^{-12}$ for the flows, and as a ground truth for evaluation, $K^*$ and $P^*$ were obtained via the \texttt{lqr} function in MATLAB using the exact model.

\begin{figure}[t]
    \centering
    \includegraphics[width=0.495\columnwidth]{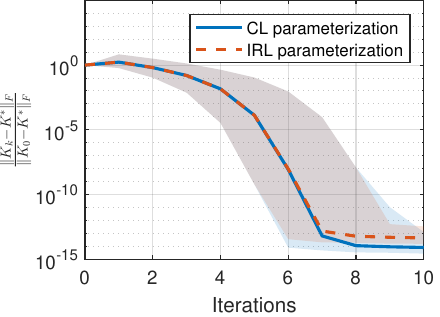}\hfill 
    \raisebox{-0.7mm}{\includegraphics[width=0.505\columnwidth]{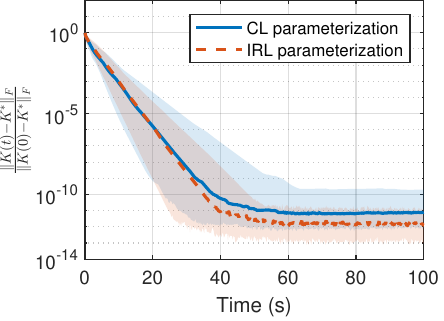}}
    \caption{Normalized residuals $\frac{\Vert K_k - K^* \Vert_F}{ \Vert K_0 - K^* \Vert_F }$ and $\tfrac{\Vert K(t) - K^* \Vert_F}{ \Vert K(0) - K^* \Vert_F }$ for the PI and policy gradient flow, respectively, under the CL and IRL parameterizations. The lines show the medians over the 100 systems and the shaded areas the ranges between the 10th and 90th percentiles.}
    \label{fig:PI_grad}
\end{figure}
For the PI, the normalized residuals  $\tfrac{\Vert K_k - K^* \Vert_F}{ \Vert K_0 - K^* \Vert_F }$ for the 100 systems under CL and IRL parameterizations are shown in the left panel of Fig.~\ref{fig:PI_grad}. 
Since both algorithms generate identical iterates $K_k$, the corresponding residuals are nearly identical, with minor differences at higher iterations caused by numerical inaccuracies.
The normalized residuals of the policy gradient flows are depicted in the right panel of Fig.~\ref{fig:PI_grad}. To match the median convergence speed of \eqref{math:proj_grad_1} and \eqref{math:flow_IRL}, we set $\alpha_{\mathrm{CL}} = 70 $ and $\alpha_{\mathrm{IRL}} = 1.5$, so that both flows reach a normalized residual of $10^{-6}$ at $t\approx\SI{21}{\second}$ in the median.
By Theorem~\ref{theo:proj_grad_CL}, the gain of \eqref{math:proj_grad_1} follows \eqref{math:proj_grad_K}, whereas \eqref{math:flow_IRL} coincides with $\dot{K}=-\alpha_{\mathrm{IRL}}\nabla f_K$ since $\hat{f}_K\equiv f_K$, so that their convergence speeds differ.
The CL parameterization yields a less accurate final solution than the IRL parameterization.
\begin{figure}[t]
    \centering
    \includegraphics[width=0.5\columnwidth]{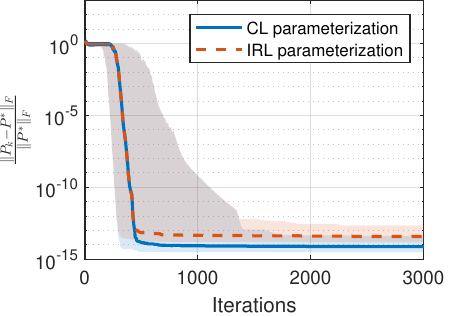}\hfill 
    \raisebox{-0.8mm}{\includegraphics[width=0.5\columnwidth]{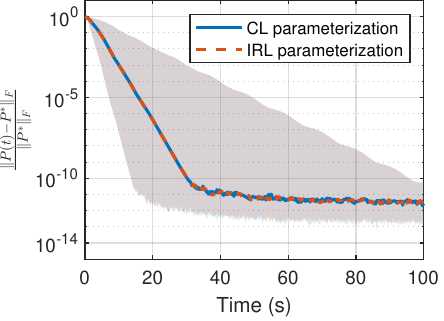}}
    \caption{Normalized Residuals $\frac{\Vert P_k - P^* \Vert_F}{ \Vert P^* \Vert_F }$, $\tfrac{\Vert P(t) - P^* \Vert_F}{\Vert  P^* \Vert_F }$ for the VI and Riccati flow, respectively, under the CL and IRL parameterizations. The lines show the medians over the 100 systems and the shaded areas the ranges between the 10th and 90th percentiles. The CL parameterization results are based on \eqref{math:CARE_closed}.}
    \label{fig:side_by_side_2}
\end{figure}
The normalized residuals with $P_0 = P(0)=0_{n,n}$ for the VI and the Riccati flow are shown in the left and right panels of Fig.~\ref{fig:side_by_side_2}, respectively. Since $P_k$ and $P(t)$ are identical for both parameterizations, their convergence rates coincide, whereas the final accuracy of the VI is higher under the CL parameterization and that of the Riccati flow is the same under both. The VI was implemented with $\varepsilon_k = 40 / (k+1)^{0.8}$ and $B_q \! = \! \{ P\succeq 0 \!  \mid \!\Vert P \Vert_F \leq 5(q+1) \}$.

\begin{figure}[t]
    \centering
    \includegraphics[scale=0.6]{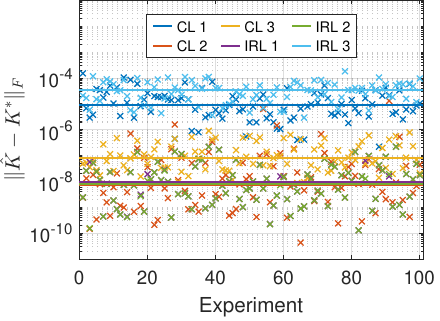}\hfill 
    \raisebox{1mm}{\includegraphics[scale=0.6]{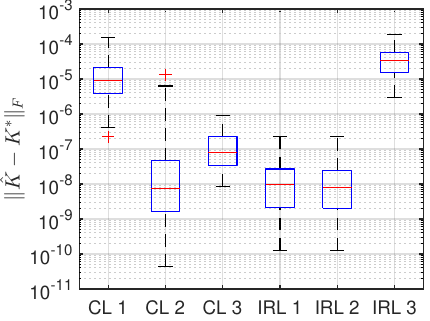}}
    \caption{Residuals $\Vert \hat{K} - K^* \Vert_F$ for the convex optimization problems \eqref{math:opt_G2} (CL 1), \eqref{math:conv_closed_2} (CL 2), and \eqref{math:closed_conv_3} (CL 3) under CL, and \eqref{math:LQR_data_2} (IRL 1), \eqref{math:IRL_conv_2} (IRL 2), and \eqref{math:IRL_conv_3} (IRL 3) under IRL parameterization.}
    \label{fig:side_by_side_3}
\end{figure}

Next, we computed the optimal feedbacks by using the convex optimization problems \eqref{math:opt_G2}, \eqref{math:conv_closed_2}, \eqref{math:closed_conv_3},  \eqref{math:LQR_data_2}, \eqref{math:IRL_conv_2}, and \eqref{math:IRL_conv_3}. As an SDP solver, we used MOSEK with default settings in YALMIP.
For each problem, the residuals for all runs and their medians are shown in the left panel of Fig.~\ref{fig:side_by_side_3}. The median accuracies of the formulations are comparable, except for \eqref{math:opt_G2} and \eqref{math:IRL_conv_3}, which are two to four orders of magnitude less accurate, although \eqref{math:opt_G2} is the most commonly used formulation in the literature. Since both are counterparts of \eqref{math:LQR_feron}, this indicates that the lower accuracy stems from the $H_2$-based formulation rather than from the parameterization. The right panel of Fig.~\ref{fig:side_by_side_3} presents box plots summarizing the stochastic performance. The residuals of \eqref{math:conv_closed_2} spread over five orders of magnitude, whereas \eqref{math:closed_conv_3}, \eqref{math:LQR_data_2}, \eqref{math:IRL_conv_2}, and \eqref{math:IRL_conv_3} exhibit more consistent behavior.

 \begin{table}
   \centering
   \caption{Median computation times for all methods}
   \label{tab:run}
   \setlength{\tabcolsep}{2pt}
\begin{tabular}{c c c c c c c c c c c}
\toprule
 & \multicolumn{2}{c}{\makecell{Iteration PI \\ in \si{\micro\second}}} & \multicolumn{6}{c}{\makecell{Convex program \\ in \si{\milli\second}}} & \multicolumn{2}{c}{\makecell{Iteration VI \\ in \si{\micro\second}}} \\
\cmidrule(lr){2-3} \cmidrule(lr){4-9} \cmidrule(lr){10-11}
& \hyperref[alg:PI_closed]{CL} & \hyperref[alg:PI_IRL]{IRL} & \hyperref[math:opt_G2]{CL 1} & \hyperref[math:conv_closed_2]{CL 2} & \hyperref[math:closed_conv_3]{CL 3} & \hyperref[math:LQR_data_2]{IRL 1} & \hyperref[math:IRL_conv_2]{IRL 2} & \hyperref[math:IRL_conv_3]{IRL 3} & \hyperref[math:CARE_closed]{CL} & \hyperref[alg:VI_IRL]{IRL} \\
\midrule
\makecell{ Comp. \\ time } & $36$ & $22$ & $6.8$ & $7.4$ & $7.7$ & $2.9$ & $2.9$ & $5.1$ & $13$ & $8.8$\\
\bottomrule
\end{tabular}
\end{table}
The median computation times per iteration in PI and VI and the solver times of the convex programs are denoted in Table~\ref{tab:run},
using labels for the convex programs consistent with Fig.~\ref{fig:side_by_side_3}.
The iterations of PI and VI require microseconds, whereas the convex programs are solved within milliseconds, where IRL 1 and IRL 2 are more than twice as fast as the CL programs.
The simulation  results should be interpreted as a proof of concept illustrating feasibility, while a comprehensive performance evaluation is beyond the scope of this work.

\section{Conclusion}
This paper links the data-driven LQR approaches of behavioral system theory and RL, two fields that have been separated by a cultural gap, by providing a unified treatment of their respective parameterizations.
Moreover, to complement each field with the approaches it lacked, we introduced PI, CARE, Riccati flow, VI, and convex program for the CL parameterization, as well as  CARE, convex programs, and gradient flow for the IRL parameterization (see Table~\ref{tab:contr}).
The results clarify the structural relationships between these approaches and highlight their distinct computational characteristics, enabling a systematic selection of methods depending on the application setting.
Future work will address robustness with respect to noise and extensions to more general control problems.

\appendix
\subsection{Proof of Lemma~\ref{lemma:pos_s}:} 
\label{app:proof1} 
 By Sylvester's equation Theorem \cite[Theorem 2.4.4.1]{horn2012matrix}, \eqref{math:Pol_eval} admits a unique solution if and only if $\sigma (\bar{X} G )\cap \sigma (- \bar{X} G ) = \emptyset$. Since $G\in\mathcal{G}$, $\bar{X} G$ is Hurwitz, 
     and the condition is satisfied. Thus, the solution to \eqref{math:Pol_eval} is $P_G = \int_0^\infty e^{ (\bar{X} G )^\top t} ( Q +  (\tilde{U} G)^\top R  \tilde{U} G) e^{\bar{X} G  t} \mathrm{d}t$. Since $Q +  (\tilde{U} G)^\top R  \tilde{U} G \succeq 0$, the integrand is positive semidefinite for all $t\geq 0$, leading to $P_G\succeq 0$.    
   Similar, if $Q\succ 0$, $P_G\succ 0$ follows. Next, assume 
    that a nonzero $z\in\mathbb{R}^n$ exists such that $P_G z =0$. Then, $ 0 = z^\top P_G z = \int_0^\infty \Vert \big(Q+(\tilde{U}G)^\top R \tilde{U} G \big)^{1/2} e^{\bar{X} G t} z  \Vert^2 \mathrm{d}t$,
   which implies $\sqrt{Q}z =0$. Observability of $\left(A,\sqrt{Q} \right)$ forces $z=0$, leading to a contradiction. Hence, $P_G \succ 0$ if $(A,\sqrt{Q})$ observable. Moreover, since $P_G\succeq 0$ for any $G\in\mathcal{G}$, $\operatorname{tr}\left(P_G  \right)=\sum_i \lambda_i(P_G)\geq 0$, implying $f_\mathcal{G}\subseteq \mathbb{R}_{\geq 0}$.\hfill \QED

\subsection{Proof of Lemma~\ref{lem:set}:} 
\label{app:proof2} 
   The set $\mathcal{K}$ \eqref{math:K} is path-connected, unbounded, and open \cite[Section 3]{bu2019}. For any $G_1,G_2 \in \mathcal{G}$ with associated $K_1,K_2 \in \mathcal{K}$, let $K(\lambda)$ for $\lambda\in[0,1]$ be a path in $\mathcal{K}$ connecting $K_1,K_2$. Define 
    $ G(\lambda) = \left[\begin{smallmatrix}\tilde{U} \\ \tilde{X}\end{smallmatrix}\right]^\dagger \Bigl[\begin{smallmatrix}-K(\lambda) \vphantom{\tilde{X}} \\ I_n \vphantom{\tilde{X}} \end{smallmatrix}\Bigr]+  (1-\lambda)N_1 + \lambda N_2$,
   where $N_i$ satisfy $\left[\begin{smallmatrix}\tilde{U} \\ \tilde{X}\end{smallmatrix}\right]N_i = 0$.  The path $G(\lambda)$ is continuous in $\lambda$, satisfies \eqref{math:data_c}, and ensures $\bar{X} G(\lambda)$ is Hurwitz because $K(\lambda)\in\mathcal{K}$.
    Hence, $G(\lambda) \in \mathcal{G}$, proving path-connectedness.
   Moreover, since $\mathcal{K}$ contains gains with $\Vert K\Vert_2\to\infty$, the minimal-norm solutions $G^p$ grow unbounded due to their linear dependence on $K$, so $\mathcal{G}$ is unbounded. For $T>m+n$, the null space $\mathcal{N}=\operatorname{ker}(\left[\begin{smallmatrix}\tilde{U} \\ \tilde{X}\end{smallmatrix}\right])$ is additionally nontrivial, so each solution set $\mathcal{G}_K=G^p+\mathcal{N}$ is itself an unbounded affine subspace.
For relative openness, consider the continuous map $\phi:\mathbb{R}^{T\times n}\to \mathbb{R}^{n\times n}$  defined by $ \phi(G)= \bar{X} G$ and the open set $\mathcal{H}=\{A \!\in \!\mathbb{R}^{n\times n}\mid \operatorname{Re}(\lambda_i (A ))\!< \!0 ~ \forall i\}$. Since $\phi$ is continuous, $\phi^{-1}(\mathcal{H})$ is open in $\mathbb{R}^{T\times n}$. Since $\mathcal{G} = S\cap \phi^{-1}(\mathcal{H})$, it is relatively open in $\mathcal{S}$. \hfill \QED

\subsection{Proof of Lemma~\ref{Lemma:noncoercive}}
\label{app:proof3} 
 Let $\{G_{i}\}\subset\mathcal{G}$ converge to $G \in \partial \mathcal{G}$, and let $K_i=-\tilde{U}G_i$ and $K=-\tilde{U}G$. By \eqref{math:G_boundary}, $\bar{X}G=A_K$ has an eigenvalue $\lambda$ with $\operatorname{Re}(\lambda)=0$. By continuity, $\bar{X}G_i=A_{K_i}$ has eigenvalues $\lambda_i\to\lambda$ whose unit eigenvectors $v_i$ converge, along a subsequence, to an eigenvector $v$ of $A_K$ associated with $\lambda$. Pre- and post-multiplying \eqref{math:Pol_eval} by $v_i^{\mathsf{H}}$ and $v_i$ yields $v_i^{\mathsf{H}}P_{G_i}v_i = (v_i^{\mathsf{H}}Qv_i+\Vert R^{1/2}K_iv_i\Vert_2^2)/(2\vert\operatorname{Re}(\lambda_i)\vert)$. The numerator tends to $v^{\mathsf{H}}Qv+\Vert R^{1/2}Kv\Vert_2^2>0$, since $Kv=0$ implies $Av=\lambda v$ and thus $\sqrt{Q}v\neq 0$ by Assumption~\ref{ass:1}. As $P_{G_i}\succeq 0$ by Lemma~\ref{lemma:pos_s}, $f_{G_i}=\operatorname{tr}(P_{G_i})=\sum_j\lambda_j(P_{G_i})\geq\lambda_{\max}(P_{G_i})\geq v_i^{\mathsf{H}}P_{G_i}v_i\to\infty$. Since this holds for a further subsequence of every subsequence, $f_{G_i}\to\infty$.
   Next, let   $G_1 = \left[\begin{smallmatrix}\tilde{U} \\ \tilde{X}\end{smallmatrix}\right]^\dagger \Bigl[\begin{smallmatrix}-K \vphantom{\tilde{X}} \\ I_n \vphantom{\tilde{X}} \end{smallmatrix}\Bigr]+ N$,
   where $\left[\begin{smallmatrix}\tilde{U} \\ \tilde{X}\end{smallmatrix}\right] N = 0$ and $K\in\mathcal{K}$. Then, %
   $P_{G_1} = \int_0^\infty e^{ (\bar{X} G_1 )^\top t}  \bigl(Q +  {K}^\top R  K \bigr) e^{\bar{X} G_1  t} \mathrm{d}t$ 
 converges since $\bar{X} G_1$ is Hurwitz. Thus, $\Vert G_1 \Vert_2 \to \infty$ as $\Vert N \Vert_2 \to \infty$, but $\lim_{\Vert N \Vert_2 \to \infty } \Vert P_{G_1} \Vert_2 < \infty$ since $P_{G_1}$ is finite in the limit, implying $\lim_{\Vert N \Vert_2 \to \infty } f_{G_1} < \infty$.
    Moreover, $f_{G_1}=f_K\to\infty$ as $\Vert K \Vert_2\to\infty$, since $f_K$ is coercive over $\mathcal{K}$.
 \hfill \QED

\subsection{Proof of Lemma~\ref{lemma:aux}}
\label{app:proof4}
 By Assumption~\ref{ass:2}, $\operatorname{im }(\tilde{U}^\top) \cap \operatorname{im }(\tilde{X}^\top) =\{0\}$. 
    Since $\Pi$ is the orthogonal projection onto $\operatorname{ker}(\tilde X)$, its kernel is $\operatorname{im}(\tilde{X}^\top)$. 
    Consequently, $\Omega = \tilde{U}\Pi$ has full row rank $m$, implying $\Omega \Omega^\dagger = I_m$ and $\Omega^\dagger = \Omega^\top (\Omega \Omega^\top)^{-1}$. Let $\Upsilon=\Pi M \Pi =\Omega^\top R \Omega$ and $\tilde \Upsilon  = \Omega^\dagger R^{-1} (\Omega^\dagger)^\top$. We verify the four Moore--Penrose conditions to show
     $\Upsilon^\dagger = \tilde \Upsilon$,
     \begin{align}
        \Upsilon \tilde\Upsilon \Upsilon & = \Omega^\top  R \Omega \Omega^\dagger R^{-1} (\Omega^\dagger)^\top \Omega^\top R \Omega = \Omega^\top  R \Omega  =  \Upsilon,\\
        \tilde\Upsilon \Upsilon \tilde\Upsilon & = \Omega^\dagger R^{-1} (\Omega^\dagger)^\top \Omega^\top R \Omega \Omega^\dagger R^{-1} (\Omega^\dagger)^\top  = \tilde\Upsilon, \\
            \Upsilon \tilde\Upsilon &= \Omega^\top R \Omega \Omega^\dagger R^{-1} (\Omega^\dagger)^\top = (\Omega^\dagger \Omega)^\top,\\
      \tilde\Upsilon \Upsilon &= \Omega^\dagger R^{-1} (\Omega^\dagger)^\top \Omega^\top R \Omega = \Omega^\dagger \Omega.
     \end{align} 
    Since $\Omega^\dagger \Omega$ is an orthogonal projection matrix, it is symmetric, implying $\Upsilon \tilde\Upsilon = (\Upsilon \tilde \Upsilon)^\top$ and $(\tilde\Upsilon \Upsilon)^\top =  \tilde\Upsilon \Upsilon$. Next, we have $\Upsilon^\dagger M = \Omega^\dagger R^{-1} (\Omega^\dagger)^\top (\tilde{U}^\top R \tilde U)$. Note that $(\Omega^\dagger)^\top \tilde{U}^\top = (\tilde{U}\Omega^\dagger)^\top$ and $\tilde{U} \Omega^\dagger = \tilde{U} \Pi \Omega^\dagger = \Omega \Omega^\dagger = I_m$. Hence, $\Upsilon^\dagger M = \Omega^\dagger R^{-1} R \tilde{U} = (\tilde{U} \Pi)^\dagger \tilde{U}. $ \hfill \QED

\printbibliography

\end{document}